\documentclass[12pt,a4paper]{amsart}
\usepackage{amscd, amssymb}
\usepackage{amsthm}
\usepackage[all]{xy}
\usepackage{tikz}
\usepackage{float}
\usetikzlibrary{arrows,backgrounds,decorations.markings}

\usepackage[colorlinks=true, pdfstartview=FitV, linkcolor=blue, citecolor=black, urlcolor=black,filecolor=magenta]{hyperref}
\usepackage{graphicx, verbatim,amsmath,amssymb,subfigure,enumerate}

\allowdisplaybreaks

\def\Dom{\operatorname{Dom}}
\def\newspan{\operatorname{span}}

\def\supp{\operatorname{supp}}

\def\supp{\operatorname{supp}}

\def\interior{\operatorname{int}}

\def\Tr{\operatorname{Tr}}

\def\corad{\operatorname{corad}}

\def\C{\mathbb{C}}

\def\R{\mathbb{R}}
\def\N{\mathbb{N}}
\def\Z{\mathbb{Z}}

\def\LL{\mathcal{L}}
\def\OO{\mathcal{O}}
\def\KK{\mathcal{K}}

\def\NN{\mathcal{N}}
\def\BB{\mathcal{B}}

\def\PP{\mathcal{P}}

\def\RR{\mathcal{R}}

\def\SS{\mathcal{S}}

\def\supp{\operatorname{supp}}

\def\rrule{\mathcal{R}}

\renewcommand{\int}{\operatorname{int}}
\renewcommand{\mid}{:}

\newcommand{\phull}{\Omega_{punc}}
\newcommand{\inverse}{^{-1}}

\newtheorem{thm}{Theorem}[section]
\newtheorem{cor}[thm]{Corollary}

\newtheorem{lemma}[thm]{Lemma}
\newtheorem{prop}[thm]{Proposition}

\theoremstyle{definition}
\newtheorem{definition}[thm]{Definition}

\newtheorem{notation}[thm]{Notation}

\theoremstyle{remark}
\newtheorem{remark}[thm]{Remark}

\newtheorem{example}[thm]{Example}

\newtheorem*{Acknowledgements}{Acknowledgements}

\numberwithin{equation}{section}

\tikzstyle{vertex}=[circle,thick]
\tikzstyle{goto}=[->,shorten >=1pt,>=stealth,semithick]

\pgfpathcircle{\pgfpoint{5cm}{0cm}} {2pt}

 \title[Fractal spectral triples on Kellendonk's $C^*$-algebra of a substitution tiling]{Fractal spectral triples on Kellendonk's $C^*$-algebra of a substitution tiling}

 \author{Michael Mampusti}
 \address{Michael Mampusti and Michael F. Whittaker \\ School of Mathematics and Applied Statistics  \\ The University of Wollongong\\ NSW  2522\\ AUSTRALIA} 
 
 \email{mm554@uowmail.edu.au, Mike.Whittaker@glasgow.ac.uk}
 
 \author{Michael F. Whittaker}
 
\curraddr{Michael F. Whittaker  \\ School of Mathematics and Statistics \\ University of Glasgow \\ 15 University Gardens \\ Glasgow G12 8QW \\ SCOTLAND}
 
 \thanks{This research was partially supported by the Australian Research Council (DP130100490).}
 

\begin{document}

\begin{abstract}
We introduce a new class of noncommutative spectral triples on Kellendonk's $C^*$-algebra associated with a nonperiodic substitution tiling. These spectral triples are constructed from fractal trees on tilings, which define a geodesic distance between any two tiles in the tiling. Since fractals typically have infinite Euclidean length, the geodesic distance is defined using Perron-Frobenius theory, and is self-similar with scaling factor given by the Perron-Frobenius eigenvalue. We show that each spectral triple is $\theta$-summable, and respects the hierarchy of the substitution system. To elucidate our results, we construct a fractal tree on the Penrose tiling, and explicitly show how it gives rise to a collection of spectral triples.
\end{abstract}

\maketitle

\section{Introduction}

A tiling of the plane is a covering of $\R^2$ by a collection of compact subsets, called \emph{tiles}, for which two distinct tiles can only meet along their boundaries. The building blocks of a tiling are the \emph{prototiles}: a finite set of tiles with the property that every other tile is a translation of some prototile. A tiling is said to be \emph{nonperiodic} if it lacks any translational periodicity. One method of producing tilings is via a substitution rule; a rule that expands each tile, and breaks it into smaller pieces, each of which is an isometric copy of an original tile. A nonperiodic substitution rule gives rise to a dynamical system, called the \emph{continuous hull}, that consists of all tilings whose local patterns appear in some finite substitution of a prototile. The continuous hull becomes a dynamical system where the homeomorphism is induced by translation. In order to associate a particularly tractable $C^*$-algebra to a nonperiodic tiling, Kellendonk \cite{Kel1} places punctures in each tile, which he then uses to define a discrete subset of the continuous hull, which we refer to as the \emph{discrete hull}.

In this paper, we define spectral triples on Kellendonk's $C^*$-algebra $A_{punc}$ associated to a tiling. The fundamental new ingredients for these spectral triples, are the recently developed fractal dual substitution tilings \cite{FWW}. Suppose $T$ is a nonperiodic substitution tiling with finite prototile set $\PP$. For each prototile $p \in \PP$, a fractal dual tiling defines a fractal tree, in fact infinitely many, on a self-similar tiling $T_p$; a tiling constructed from the substitution rule on $p$. Each of our fractal trees defines a unique fractal path between the punctures of any two tiles in $T_p$. Moreover, each fractal tree on $ T_p $ respects the hierarchy of the substitution rule. Given a fractal tree on $ T_p $, we apply Perron-Frobenius theory to the substitution matrix associated to the edges of the fractal dual tiling, to define a length function on each fractal edge in the fractal tree. This extends to a self-similar length function on the entire fractal tree, with scaling factor given by the Perron-Frobenius eigenvalue $\kappa$. If $ \lambda $ is the scaling factor for the original tiling, the scaling factor $\kappa$ of the fractal tree, is related to the Hausdorff dimension $h$ of the fractal dual tiling by the formula $h = \frac{\ln \kappa}{\ln \lambda}$. The fractal tree is then used to define a length function between any two tiles of $T_p$ using Perron-Frobenius theory. Let $d_{\mathfrak{F}_p}(t,t')$ denote the fractal length between the punctures of two tiles $t$ and $t'$ in $T_p$.

To each substitution tiling with a fractal dual tiling, we construct spectral triples on Kellendonk's $C^*$-algebra $A_{punc}$, which we now outline. For each $p \in \PP$, let $H_p:=\ell^2(T_p \setminus \{p\})$, with canonical basis $\{\delta_t \mid t \in T_p \setminus \{p\}\}$, and define an unbounded multiplication operator $D_p \delta_t:=\ln(d_{\mathfrak{F}_p}(t,p))\delta_t$. We show that $(A_{punc},H_p,D_p)$ is a $\theta$-summable (positive) spectral triple. Let $H:=\oplus_{p \in \PP} H_p$. For each function $\sigma:\PP \to \{-1,1\}$, we define an unbounded multiplication operator $D_{\sigma}:=\oplus_{p \in \PP} \sigma(p)D_p$. Then, $(A_{punc},H,D_{\sigma})$ is also a $\theta$-summable spectral triple. This defines a collection of spectral triples on Kellendonk's algebra $A_{punc}$ that each respect the hierarchy of the substitution rule.

Using operator algebras as the basic framework, Alain Connes developed noncommutative geometry \cite{Con}, and has shown its significance to many fields of mathematics. In particular, one of the overarching themes of noncommutative geometry is to describe a consistent mathematical model for quantum physics. Dynamical systems are particularly well suited to the tools of noncommutative geometry, and provide dynamical invariants in a noncommutative framework. Of particular importance to Connes' program are spectral triples, which typically define a noncommutative Riemannian metric on a $C^*$-algebra. A spectral triple $(A,H,D)$ consists of a $C^*$-algebra $A$ faithfully represented on a separable Hilbert space $H$, and a self-adjoint unbounded operator $D$ on $H$ with compact resolvent, whose commutators with a dense $*$-subalgebra of $A$ are bounded.

The noncommutative topology of tilings has a long history. Alain Connes initiated the study of substitution tilings in a noncommutative framework by giving a detailed description of a $C^*$-algebra associated with the Penrose tiling in his seminal book \cite{Con}. In 1982, Dan Shechtman discovered quasicrystals \cite{She}, a type of material that is neither crystalline nor amorphous. The mathematical theory explaining Shechtman's discovery had already been developed in the context of purely mathematical research; nonperiodic tilings provide an excellent model for quasicrystals.
In an attempt to understand the physics of quasicrystals, Bellissard defined a crossed product $C^*$-algebra by a family of Schr\"odinger operators. Years later, Kellendonk defined a discrete version of the continuous hull and constructed a groupoid $C^*$-algebra associated with a tiling \cite{Kel1, Kel2}. Soon afterwards, Anderson and Putnam \cite{AP} showed that the continuous hull $\Omega$ of a tiling is a Smale space, and used this observation to describe the $K$-theory of the crossed product $C(\Omega) \rtimes \R^2$. More recently, Kellendonk's construction was generalised to tilings with infinite rotational symmetry in \cite{Whi_rot}, and the rotationally equivariant $K$-theory of these algebras was completely worked out in \cite{Star_Kth}. 

Only recently has there been a breakthrough in the noncommutative geometry of tilings. The primary interest in spectral triples on tilings is that the continuous hull of a nonperiodic tiling is not only a topological object, it also has rich geometric structure. The groundbreaking spectral triple for tilings appeared in John Pearson's 2008 thesis \cite{Pearson2008}, and the subsequent joint paper with Bellissard \cite{BP}. These spectral triples were defined on the commutative $C^*$-algebra associated with the hull of a tiling. A few years later, the second author constructed spectral triples on the unstable $C^*$-algebra of a Smale space \cite{Whi_thesis,Whi_ST}, which is strongly Morita equivalent to Kellendonk's algebra. 
However, in the special case of tiling algebras, this spectral triple essentially measured the Euclidean distance between two tilings in the groupoid used to define the $C^*$-algebra, and ignored the substitution system. Since Bellissard and Pearson's seminal result there have been a number of papers on spectral triples of tilings, see for example \cite{JS,JS2,KLS,KS}. The survey article \cite{JKS} explains these constructions and their relationship to one another.

\begin{Acknowledgements}
We thank the anonymous reviewer for suggesting revisions that greatly improved the exposition.
\end{Acknowledgements}

\section{Nonperiodic tilings and their properties} \label{sec:preliminaries}

The tilings in this paper are built from \emph{prototiles}, a finite collection $\PP$ of labelled compact subsets of $\R^2$ that each contain the origin, and are equal to the closure of their interior. We denote the label of a prototile $p \in \PP$ by $\ell(p)$, the support of $p$ by $\supp(p) \subset \R^2$, and the boundary of $ p $ by $ \partial \supp (p) $. In general, given a subset $ X \subset \R^2 $, we write $ \partial X $ for the boundary of $ X $. The labels allow us to have two distinct prototiles with the same support, and we often denote the labels by colours. A \emph{tile} is defined to be any translation of a prototile. So, for any $p \in \PP$ and $x \in \R^2$, the labelled subset $t:=p+x$ is a tile with label $\ell(t) := \ell(p)$, and support $\supp(t) := \supp(p)+x$.

\begin{definition} \label{defn:tiling}
Let $\PP$ be a set of prototiles. A \emph{tiling of the plane} is a countable collection $ T $ of tiles, each of which is a translate of some prototile $ p \in \PP $, such that
\begin{enumerate}[(1)]
\item $\cup_{t \in T}\ \supp(t) = \R^2$; and
\item $\interior(\supp(t)) \cap \interior(\supp(t'))=\varnothing$ whenever $t \neq t'$.
\end{enumerate}
A tiling $T$ is said to be \emph{edge-to-edge} if whenever two tiles intersect, they meet full edge to full edge, or at a common vertex (see \cite[Section 3.2]{FWW} for the definition of an edge in the case that the tiling does not have polygonal prototiles). If tiles in an edge-to-edge tiling $T$ only intersect along at most one edge, then $T$ is said to be \emph{singly edge-to-edge}. A \emph{patch} $P \subset T$ is a finite collection of tiles in $T$ such that the interior of the support of $P$ is connected.
\end{definition}

For $x \in \R^2$ and $r>0$, let $B(x,r)$ denote the ball of radius $r$ centred at $x$.
Given a tiling $T$, we require the following collections of tiles. For $x \in \R^2$ and $r > 0$, let
\[
T \sqcap B(x,r) := \{t \in T : t \subset B(x,r)\}.
\]
Then, $T \sqcap B(x,r)$ is the collection of tiles $t \in T$ whose support is completely contained in $B(x,r)$.

Suppose $T$ is a tiling, and $x \in \R^2$. The \emph{translate} of $T$ by $x$ is $T+x := \{t+x : t \in T\}$.
The \emph{orbit} of $T$ is the set
$\OO(T) := \{T+x : x \in \R^2\}$.
If there exists $x \in \R^2 \setminus \{0\}$ such that $T+x = T$, we say that $T$ is \emph{periodic}. If $T+x=T$ implies that $x=0$, then we say that $T$ is \emph{nonperiodic}.

\begin{definition}\label{tiling metric}
Let $T$ be a tiling. For $T_1,T_2 \in \OO(T)$ the \emph{tiling metric} is given by
\[
d(T_1,T_2) := \inf \{\varepsilon, 1 \mid (T_1-x_1) \sqcap B(0, \varepsilon^{-1}) = (T_2-x_2) \sqcap B(0, \varepsilon^{-1}),\ x_1, x_2 \in \R^2,\ |x_1|, |x_2| < \varepsilon\}.
\]
\end{definition}

Two tilings $T,T'$ are close in the tiling metric if $T$ and $T'$ have the same pattern of tiles on a large ball centred about the origin, up to a small translation.

\begin{definition}
The \emph{continuous hull} of a tiling $ T $, denoted by $ \Omega $, is the completion of $ \OO (T) $ in the tiling metric. If every tiling in $\Omega$ is nonperiodic, then we say that $T$ is \emph{strongly nonperiodic}.
\end{definition}

It is straightforward to check that the limit of a Cauchy sequence of tilings in $\OO(T)$ is a tiling with the same prototile set as $T$, see \cite[Section 2.1]{Kel1} for details.

In this paper, the tilings we investigate are those whose continuous hull is a compact topological space. In \cite{RW}, Radin and Wolff give a characterisation of such tilings. A tiling $T$ has \emph{finite local complexity} if there are only a finite number of two tile patches in $T$, up to translation.
In this paper, we are concerned with tilings built from finite prototile sets, whose tiles meet singly edge-to-edge.
Thus, the tilings we consider here have finite local complexity.

\begin{thm}[{\cite[Lemma 2]{RW}}] \label{thm:compact}
A tiling $T$ has finite local complexity if and only if $\Omega$ is compact.
\end{thm}

\begin{definition}
Let $\PP$ be a set of prototiles. A \emph{substitution rule} on $\PP$ is a map $\omega$ on $ \PP $, with a scaling factor $\lambda>1$, such that for each $p \in \PP$, $\omega(p)$ is a patch satisfying the following properties:
\begin{enumerate}[(1)]
\item $\cup_{t \in \omega(p)}\ \supp(t) = \lambda \supp(p)$; and
\item $\interior(\supp(t)) \cap \interior(\supp(t')) = \varnothing$ for all $t,t' 
\in \omega(p)$ with $t \neq t'$.
\end{enumerate}
\end{definition}

Suppose $\omega$ is a substitution rule on a set of prototiles $\PP$ with scaling factor $\lambda$. Let $p \in \PP$, $x \in \R^2$, and consider the tile $t=p+x$. We extend $\omega$ to tiles by $\omega(t) := \omega(p) + \lambda x$, and hence, to tilings by $\omega(T) := \cup_{t \in T}\ \omega(t)$. A similar formula holds for a patch $P$. In particular, for each $p \in \PP$, we may apply $\omega$ to the patch $\omega(p)$. We denote $\omega(\omega(p))$ by $\omega^2(p)$, and inductively define $ \omega^n (p) := \omega (\omega^{n-1} (p)) $ for each $ n \in \N $. A \emph{supertile of order $n$} is a translate of the patch $\omega^n(p)$ for some $p \in \PP$.

\begin{example}[The Penrose tiling] \label{eg:penrose}
The version of the Penrose tiling we consider, consists of the four prototiles depicted in Figure \ref{fig:penrose}, and all rotations by multiples of $\pi/5$. This gives a total of forty prototiles. The substitution rule $\omega$ applied to each of the prototiles is also illustrated in Figure \ref{fig:penrose}, and extends to all forty tiles by rotation. The scaling factor for this substitution is the golden ratio, $\lambda= (1+\sqrt{5})/2$. A patch of the Penrose tiling appears in Figure \ref{fig:penrose_patch}.
\begin{figure}[h!]
\centering
\includegraphics[width=0.3\textwidth]{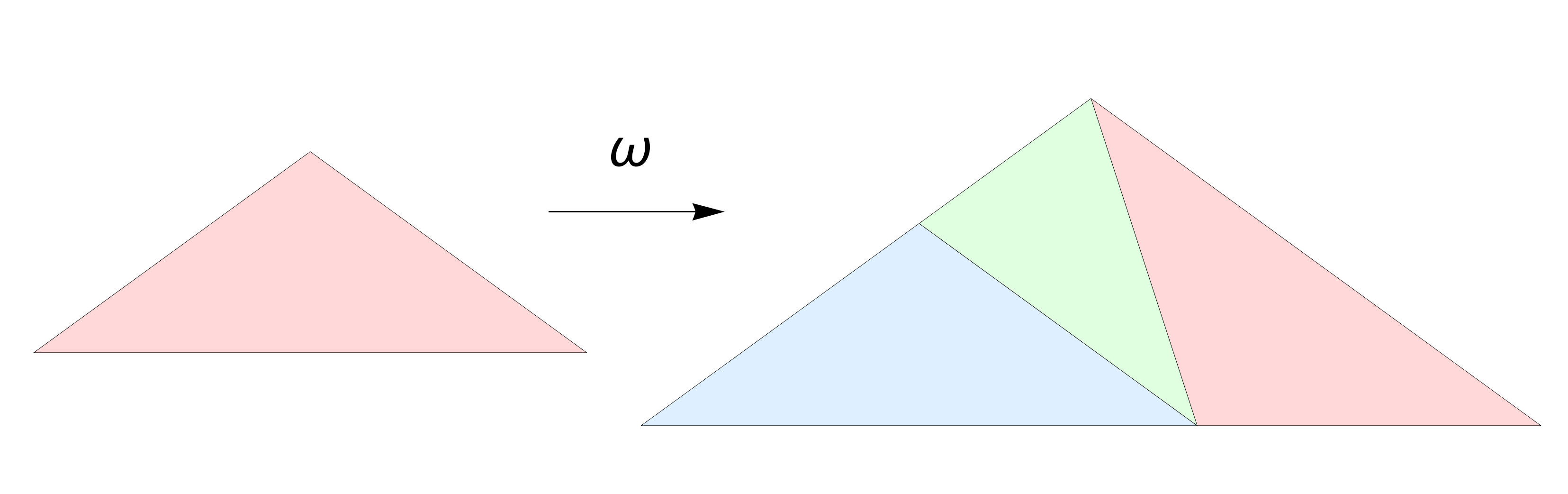}\hfill
\includegraphics[width=0.3\textwidth]{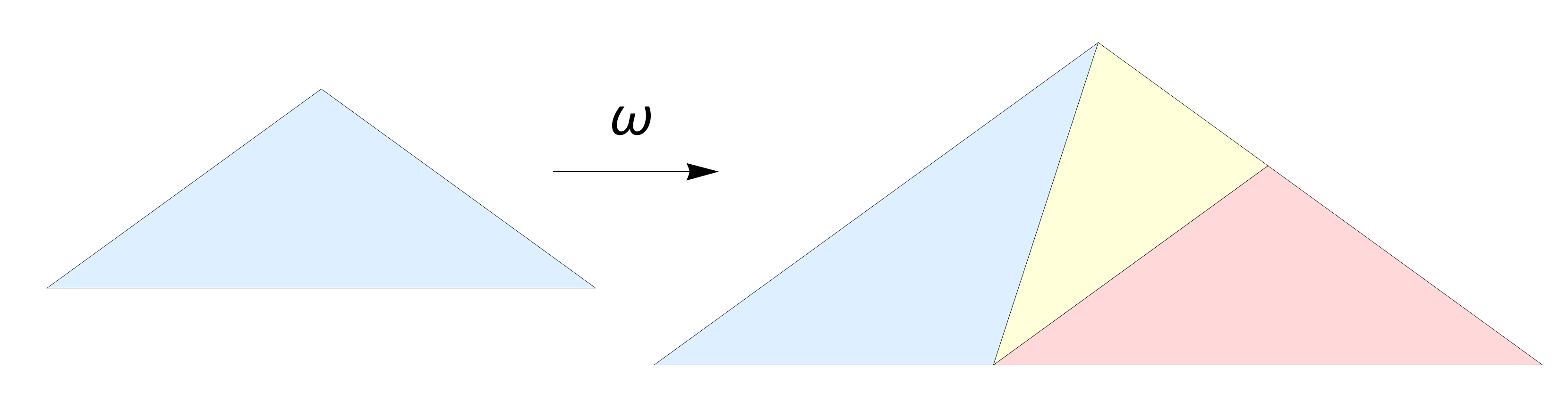}\hfill
\includegraphics[width=0.15\textwidth]{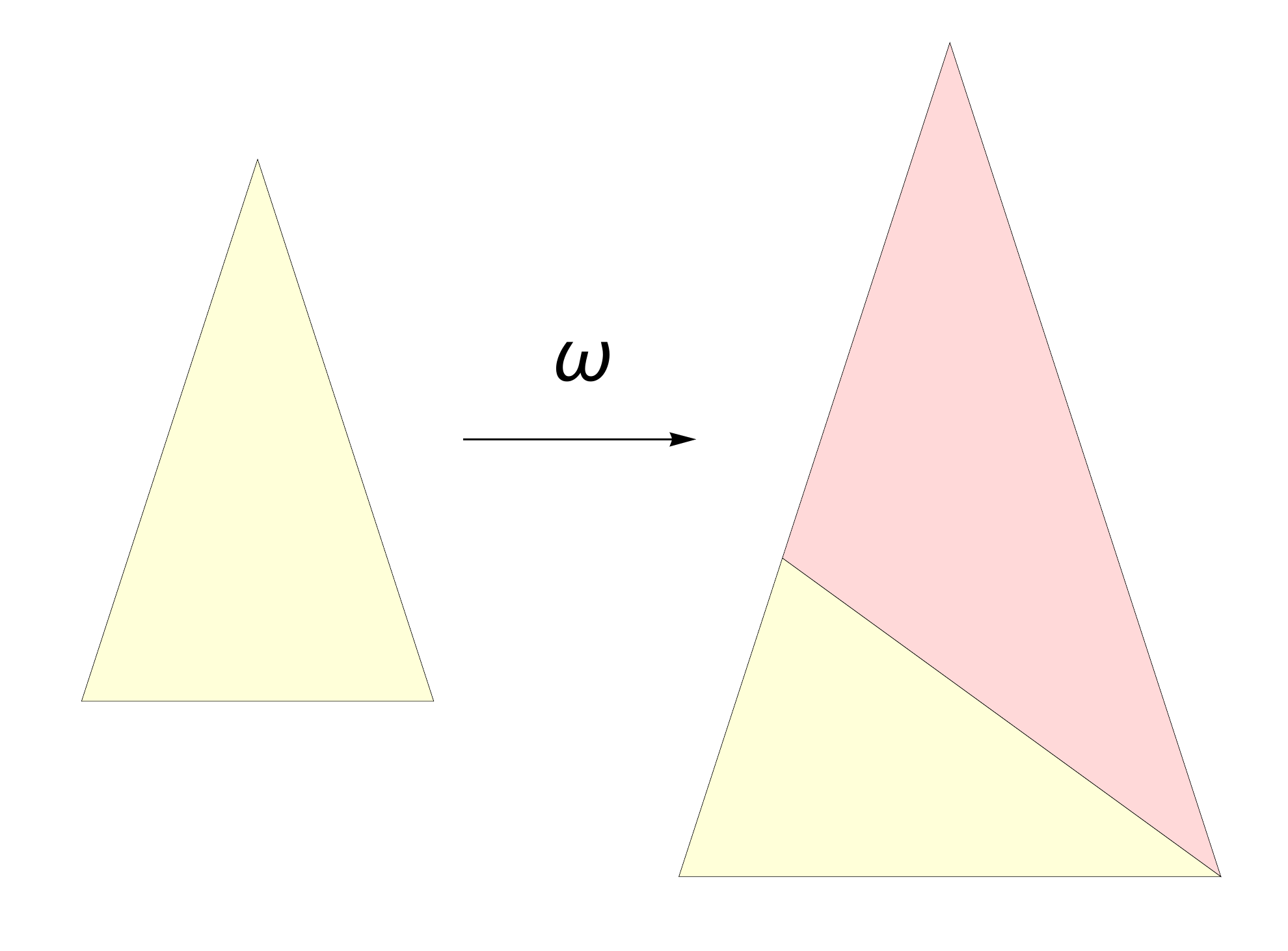}\hfill
\includegraphics[width=0.15\textwidth]{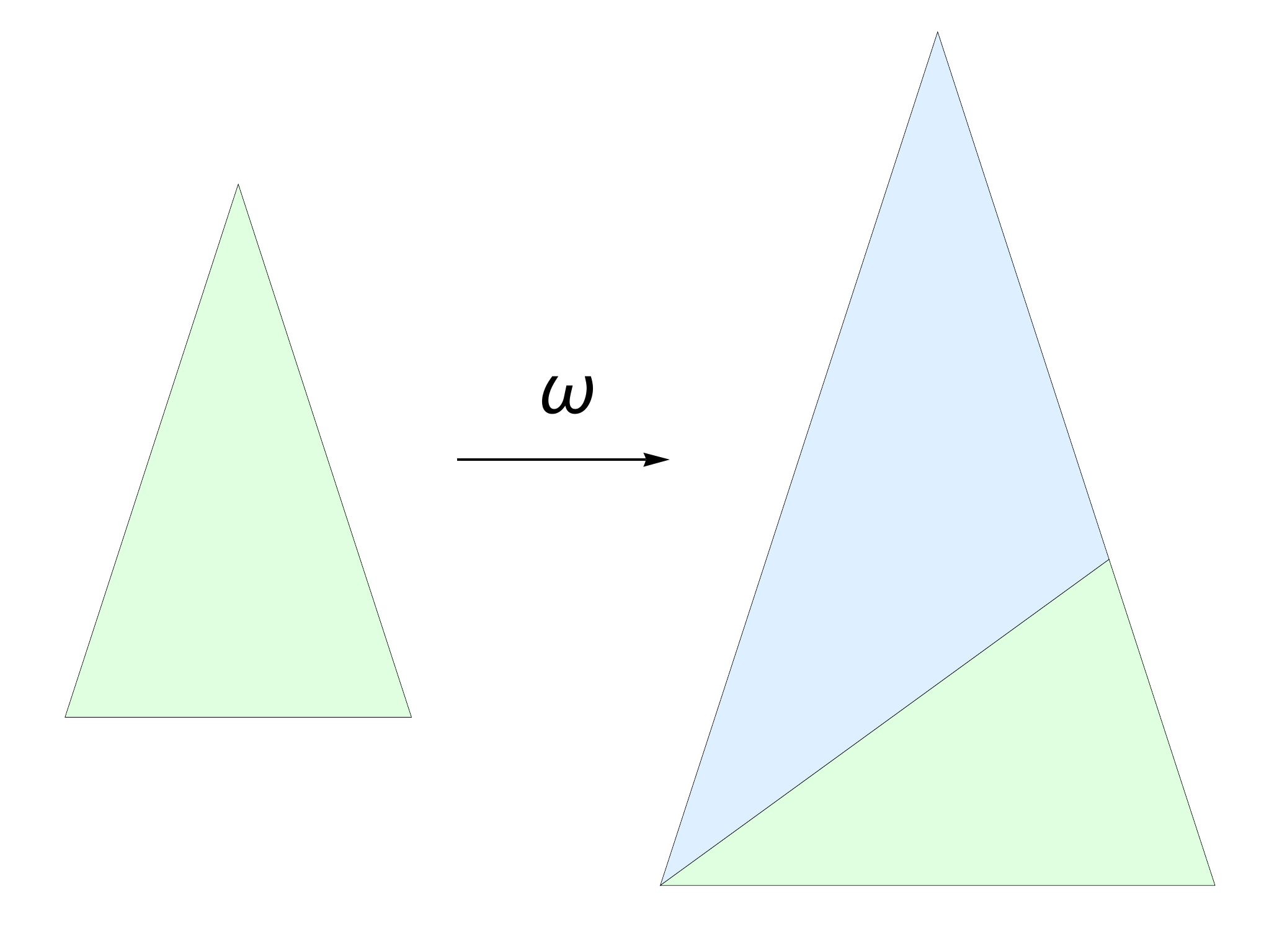}
\caption{The Penrose substitution rule} \label{fig:penrose}
\end{figure}
\begin{figure}[h!]
\centering
\frame{\includegraphics[trim=1100 250 1100 550,clip,width = 0.9\textwidth]{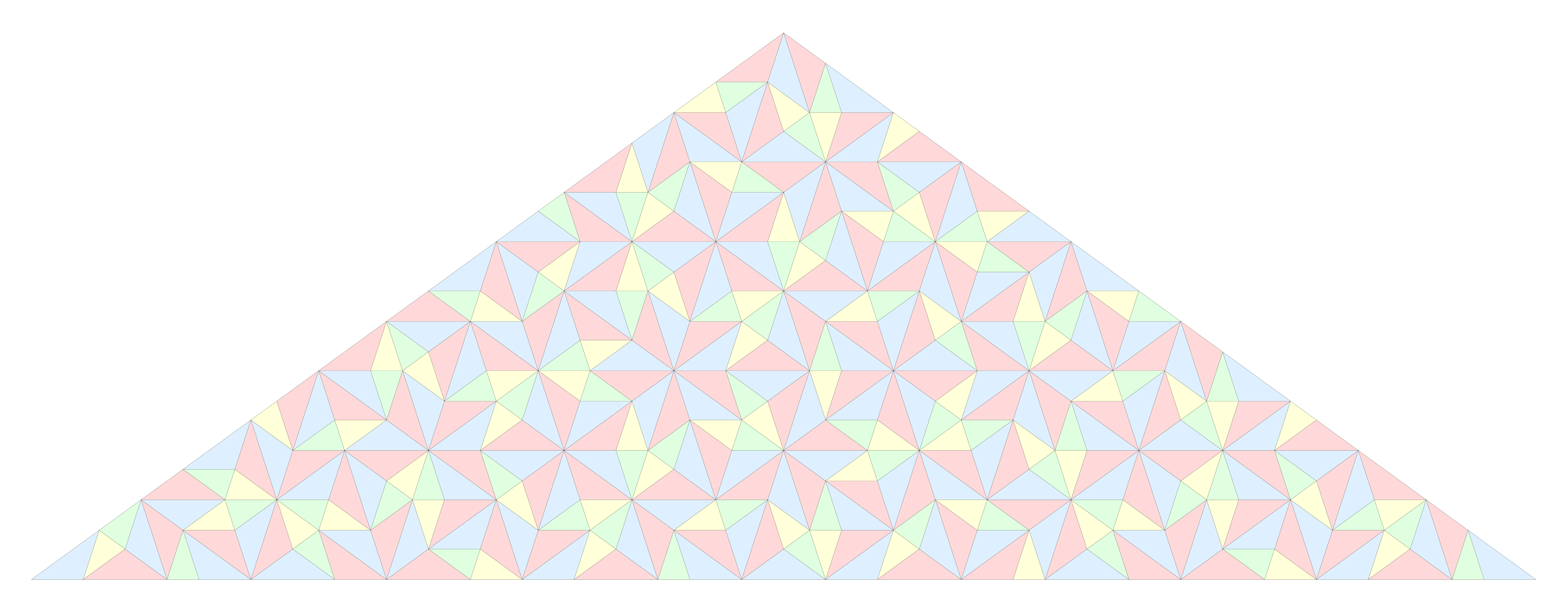}}
\caption{A patch of the Penrose tiling} \label{fig:penrose_patch}
\end{figure}
\end{example}

\begin{notation} \label{notation:substitution}
Let $ \omega $ be a substitution rule on a set of prototiles $ \PP $. For each $p \in \PP$ and $n \in \N$, let $\NN_n(p)$ denote the number of tiles in the patch $\omega^n(p)$. We write $\NN_{\max} := \max\{\NN_1(p) : p \in \PP \}$. We note that $\NN_{\max}$ is well defined, since prototile sets are finite. Moreover, for each $p \in \PP$ and $n \in \N$, it is clear that $\NN_n(p) \leq \NN_{\max}^n$.
\end{notation}

\begin{definition} \label{defn:submatrix}
Let $\omega$ be a substitution rule on a set of prototiles $\PP$. For each pair of prototiles $p,q \in \PP$, define the following set:
\[
D_{pq} := \{ x \in \R^2 : q+x \in \omega(p) \}.
\]
The matrix $D := (D_{pq})$ is called the \emph{digit matrix for $\omega$}. For each pair of prototiles $p,q \in \PP$, let $M_{pq} := |D_{pq}|$, so that $M_{pq}$ is the number of translates of $q$ in the patch $\omega(p)$. The matrix $M := (M_{pq})$ is called the \emph{substitution matrix for $\omega$}.
\end{definition}

\begin{definition}
Let $\omega$ be a substitution rule on a set of prototiles $\PP$, and let $M$ be the substitution matrix for $\omega$. If $ M $ is irreducible, in the sense that there is an $ N \in \N $ sufficiently large such that $ (M^N)_{p,q} > 0 $ for all pairs $ p,q \in \PP $, we say that $ \omega $ is a \emph{primitive substitution rule}.
\end{definition}

Following Sadun \cite[Theorem 1.4]{Lorenzo.book}, we obtain a tiling from a primitive substitution rule $\omega$ on $\PP$ in the following way. Fix a prototile $p \in \PP$. Since $\omega$ is primitive, there exists $n \in \N$ sufficiently large such that there exists a translate of $p$, say $t$, in the patch $\omega^n(p)$ whose support is contained in the interior of the support of $ \omega^n (p) $. Thus, there is a fixed point in $t$ with respect to the map from $\lambda^n p$ to $t \in \omega^n(p)$. If we place the origin of $\R^2$ at this fixed point, we obtain the sequence
\[
t \subset \omega^n(t) \subset \omega^{2n}(t) \subset \omega^{3n}(t) \subset \cdots.
\]
The union $T := \cup_{k \in \N}\ \omega^{kn}(t)$ is a tiling of $\R^2$. We may modify the prototile set and the substitution rule so that $p \subset \omega(p)$ for each $p \in \PP$. This implies that the fixed points under the substitution rule are at the origin. We make this assumption for the remainder of this paper.

\begin{definition} \label{tilinggeneratedbyp}
Let $\omega$ be a primitive substitution rule on a set of prototiles $\PP$. For each $p \in \PP$, the \emph{tiling generated by $p$} is the tiling $T_p := \cup_{n \in \N}\ \omega^n(p)$.
\end{definition}

For each $p \in \PP$, the tiling generated by $p$ is a self-similar tiling in the sense of \cite[Definition 2.5]{FWW}, since $\omega(T_p) = T_p$. In the following sections, we will often require the map $\lambda \inverse \omega$, which applies the substitution rule to prototiles, and scales the resulting patch by $\lambda \inverse$. For each $p,q \in \PP$, and each non-empty subset $S \subset \R^2$, we define the following notation:
\[
S+D_{pq} := \begin{cases}
\ \bigcup_{x \in D_{pq}}\ (S+x) &\mbox{if } D_{pq} \neq \varnothing \\
\ \varnothing &\mbox{otherwise}
\end{cases}.
\]
Note that if we identify each prototile with its support, then for each $ p \in \PP $, we have
\[
\omega(p) = \bigcup_{q \in \PP}\ q+D_{pq}.
\]

From the digit matrix, we build a map $\RR$. Consider the set of non-empty compact subsets of $\R^2$, denoted by $H(\R^2)$. The domain of $\RR$ is given by $X := \sqcup_{p \in \PP}\ H(\R^2)$, and for $B \in X$, we write $B_p$ for the $p^{\text{th}}$ coordinate in the disjoint union. Then, for any $B \in X$ and $p \in \PP$, we define
\[
\RR(B)_p := \bigcup_{q \in \PP}\ \lambda \inverse(B_q+D_{pq}),
\]
and set $\RR(B) := \sqcup_{p \in \PP}\ \RR(B)_p$. We note that $ \RR (B) \in X $ for each $ B \in X $ so that $ \RR : X \to X $. The map $\RR$ is called the \emph{contraction map for $\omega$}, and can be found in \cite[Definition 4.2]{FWW}.

Let $ \widehat{\PP} $ denote the disjoint union of the supports of each prototile $ p \in \PP $. Then, $ \widehat\PP$ is an element of $X$, and it is straightforward to check that $\RR(\widehat\PP)=\widehat\PP$. In the sequel, we will often abuse notation, and identify the prototile set $ \PP $ with the corresponding supports $\widehat\PP$. Then, it is implicitly understood that these compact subsets of $X$ carry their label. For example, $\RR(\PP)_p$ is the collection of labelled compact subsets $\lambda^{-1} \omega(p)$ in $X_p$. Moreover, $\RR:X \to X$ can be iterated, and for each $ n \in \N $, we call each scaled down tile $t \in \RR^n(\PP)_p$ an \emph{$n$-subtile}, and each $t \in \RR(\PP)_p$ a \emph{subtile}.

\subsection{The discrete hull} \label{sec:discretehull}

In this subsection, we look at a particular subset of the continuous hull of a tiling, called the \emph{discrete hull}. Let $T$ be a tiling with prototile set $\PP$. Following Kellendonk \cite[Section 2.1]{Kel1}, for each $p \in \PP$, we choose a point in the interior of $p$. This point is called the \emph{puncture} of $p$, denoted $x(p)$. For any vector $y \in \R^2$, we choose the puncture of the tile $t := p+y$ to be $x(t):=x(p)+y$. In \cite[Section 2.1]{Kel1}, it is shown that each tile $t$ of every tiling in the continuous hull $\Omega$, can be assigned a puncture that is consistent in the sense that if $t := p+y$, then $x(t)=x(p)+y$.

\begin{definition}[{\cite[Definition 1]{Kel1}}] \label{defn:discretehull}
The \emph{discrete hull} of a tiling $ T $, is given by
\[
\phull := \{T' \in \Omega :\ \text{there exists}\ t \in T'\ \text{with}\ x(t) = 0 \} \subset \Omega.
\]
\end{definition}

The tilings in the discrete hull are those that contain a tile whose puncture is at the origin of $\R^2$. For each $T' \in \phull$, let $T'(0)$ denote the unique tile $t \in T'$ whose puncture is the origin. In the relative topology inherited from $\Omega$, and assuming finite local complexity, the discrete hull $\phull$ is a Cantor set. In particular, $\phull$ is a compact metric space which has a basis of both open and closed sets. Indeed, given a patch $P$ and a tile $t$ in $P$, the set
\[
U(P,t) := \{ T' \in \phull : P-x(t) \subset T' \}
\]
is both open and closed in $\phull$, and the set of all such sets forms a basis for the topology on $\phull$. Let $T',T'' \in \phull$. We say that $T'$ and $T''$ are \emph{translation equivalent} if there exists $t \in T'$ such that $T'-x(t) = T''$. We denote by $[T']$ the translation equivalence class of $T'$.

Let $\omega$ be a primitive substitution rule on a set of prototiles $\PP$. For the remainder of this paper, we assume that $x(p) = 0$ for each $p \in \PP$, so that each puncture is a fixed point under the substitution rule. The assumptions we have made on the substitution rule $\omega$, and the set of prototiles $\PP$, force $T(0) \in \PP$ for all $T \in \phull \subset \Omega$. Furthermore, for each $p \in \PP$, $x(p) = 0$ implies $T_p \in \phull$.

\section{Fractal dual tilings and fractal trees}\label{sec:fractal_trees}

In \cite{FWW}, the authors construct fractal dual substitution tilings from a given substitution tiling. At the core of their construction, is the notion of a \emph{recurrent pair}. In this section, we give a basic construction using dual trees, rather than quasi-dual trees, of the fractal dual substitution tilings described in \cite{FWW}. We note that the constructions in this paper work in full generality of quasi-dual trees, and are stated as such. In each definition, we refer the reader to the more general definition for quasi-dual trees found in \cite{FWW}. We begin with the definition of a consistent dual tree.

\begin{definition}[cf. {\cite[Definition 3.6 and 3.7]{FWW}}]
Let $T$ be a substitution tiling with finite local complexity, and prototile set $\PP$. A \emph{dual tree} $G_p$ in a prototile $p \in \PP$, consists of a vertex $v_p$ in the interior of $p$, and a collection of non-overlapping edges connecting $v_p$ to a boundary vertex in the interior of each edge of $p$. Considering $ \PP $ as an element of $ X $, we say that $G = \sqcup_{p \in \PP}\ G_p$ is a \emph{consistent dual tree} in $ \PP $, if $G_p$ is a dual tree in $p$ for each $ p \in \PP $, and if two translated prototile edges meet in the tiling $T$, then the associated boundary vertices of $G$ meet in $T$ as well.
\end{definition}

\begin{notation}
	Let $ T $ be a tiling with finite local complexity, and prototile set $ \PP $. Suppose that $ G = \sqcup_{p \in \PP}\ G_p $ is a consistent dual tree in $ \PP $. For each $ p \in \PP $, we denote the set of edges in $ G_p $ by $ E (G_p) $, the set of vertices in $ G_p $ by $ V (G_p) $, and $ \partial V (G_p) $ for those vertices which lie on $ \partial \supp (p) $. Moreover, for each $ n \in \N $, we write $ E (\RR^n (G)_p) $ for the images of edges in each $ G $ under $ \RR^n $.
\end{notation}

\begin{definition}[cf. {\cite[Definition 5.1]{FWW}}]
Let $T$ be a substitution tiling with finite local complexity, substitution rule $ \omega $, and prototile set $ \PP $. A pair of consistent dual trees $(G,S)$ on $\PP$, is called a \emph{recurrent pair} if $S \subset \rrule^n(G)$ for some $n \in \mathbb{N}$.
\end{definition}

Under certain conditions, developed in detail in \cite[Section 7.1]{FWW}, the attractor of a recurrent pair $(G,S)$ is a consistent dual tree $A$ inscribed in the prototile set $\PP$. Since the graph $A$ is the attractor of the graph iterated function system defined by equation (7.2) in \cite[Section 7.1]{FWW}, we call $A$ a \emph{fractal graph}. Moreover, $A$ is invariant under this graph iterated function system. It is well known that embedding a consistent dual tree into each prototile of a tiling yields a new tiling, often called the combinatorial dual of $T$. In this case, if we embed the dual trees in $A$ into each tile of $T$, we obtain a new tiling that is also a substitution tiling, since $A$ is built from the substitution of $T$.

\begin{thm}[{\cite[Theorems 6.6 and 6.9]{FWW}}]\label{thm6.6-FWW}
Let $T$ be a substitution tiling with finite local complexity, whose tiles meet singly edge-to-edge. Then $T$ has an infinite number of distinct fractal quasi-dual substitution tilings. If the prototiles of $T$ are all convex, then $T$ has an infinite number of distinct fractal dual substitution tilings.
\end{thm}

\begin{example}[A fractal version of the Penrose tiling]\label{penrose_recurrentpair}
Recall the Penrose tiling constructed in Example \ref{eg:penrose}. The trees giving rise to a recurrent pair $(G,S)$ are depicted in Figure \ref{eg:penrosefractal} on two prototiles, and extend uniquely to all forty prototiles by rotation and reflection. The unique fractal graph $A$ defined by the recurrent pair $(G,S)$ is depicted on the right hand side of Figure \ref{eg:penrosefractal}, and also extends to all forty prototiles.

\begin{figure}[h!]
\centering
\begin{tikzpicture}[>=stealth,scale=1.1]
\node (a) at (0,4) [label=below:{$G$}] {\includegraphics[width=3cm]{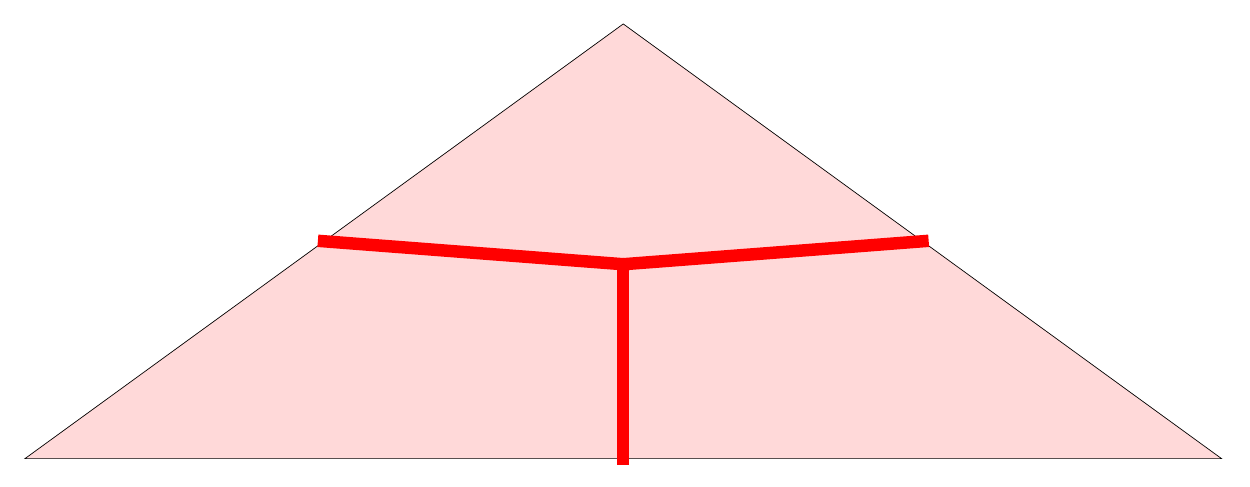}};
\node (b) at (3.5,4) [label=below:{$\rrule^2(G)$}] {\includegraphics[width=3cm]{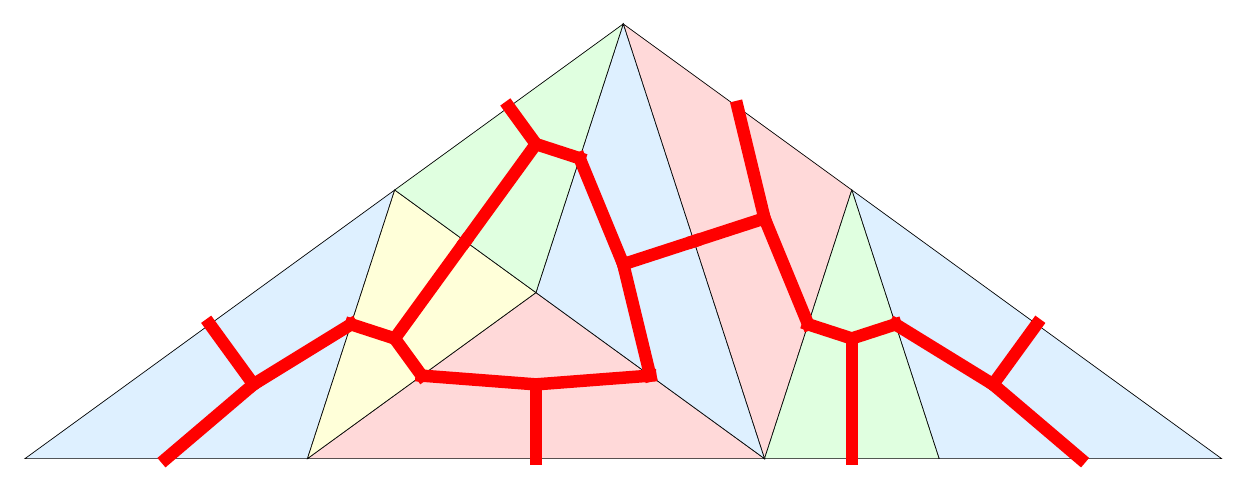}}
	edge[<-] (a);
\node (c) at (7,4) [label=below:{$S$}] {\includegraphics[width=3cm]{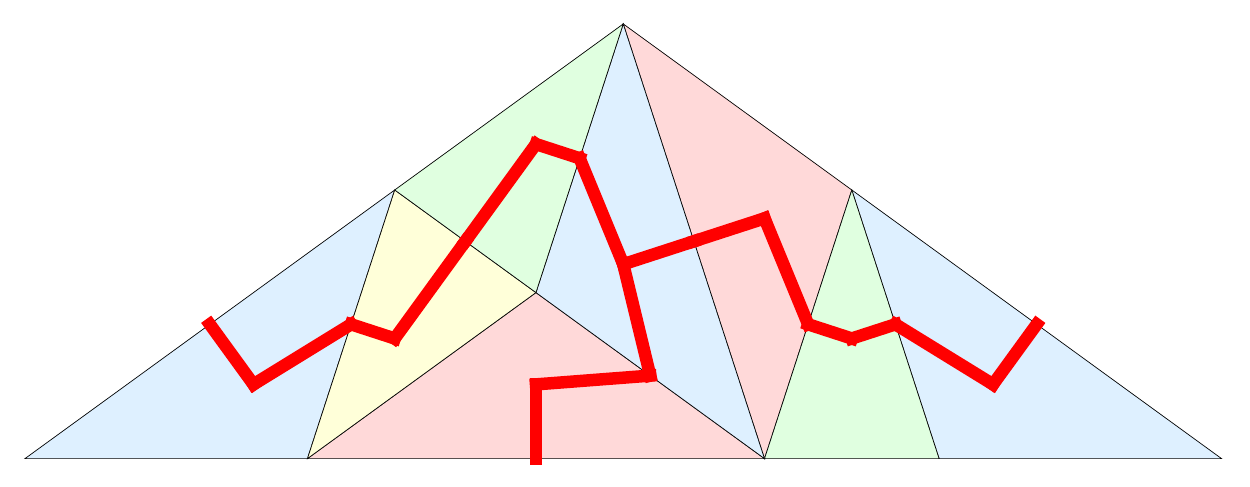}}
	edge[<-] (b);
\node (d) at (10.5,4) [label=below:{$A$}] {\includegraphics[width=3cm]{penrose_lim_1}}
	edge[<-,dashed] (c);
	
\node (e) at (1,1.5) [label=below:{$G$}] {\includegraphics[width=1.5cm]{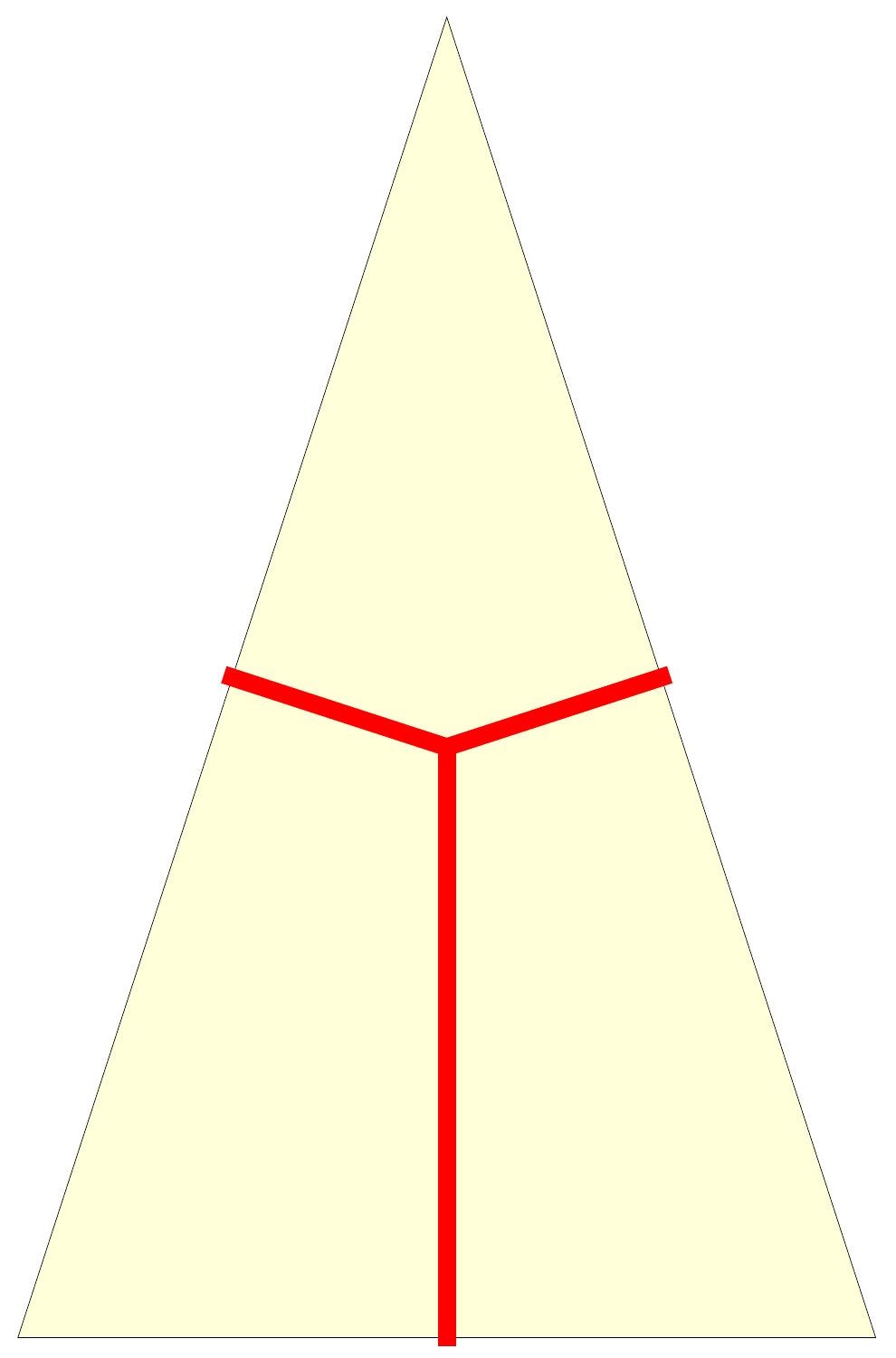}};
\node (f) at (4,1.5) [label=below:{$\rrule^2(G)$}] {\includegraphics[width=1.5cm]{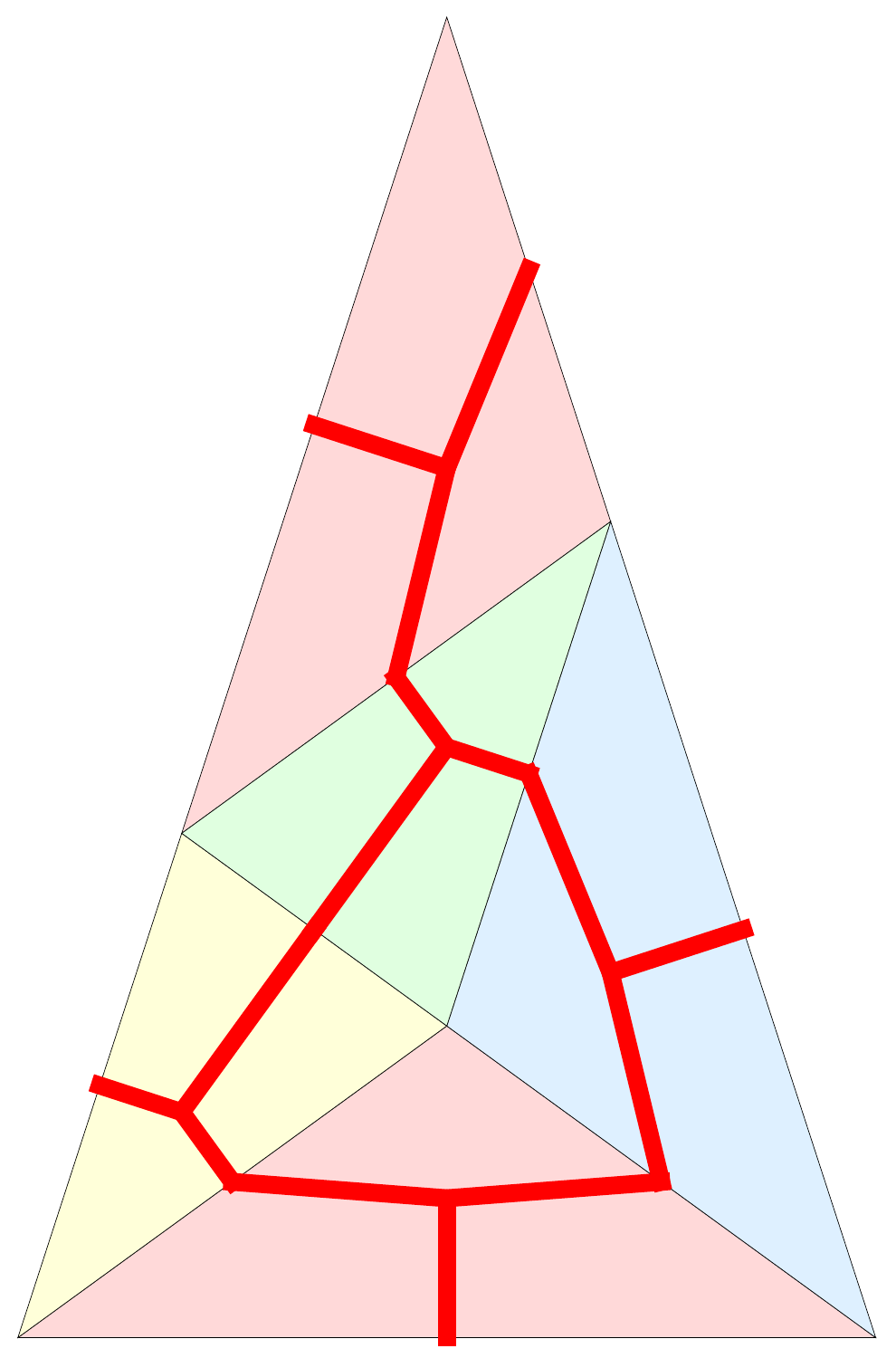}}
	edge[<-] (e);
\node (g) at (7,1.5) [label=below:{$S$}] {\includegraphics[width=1.5cm]{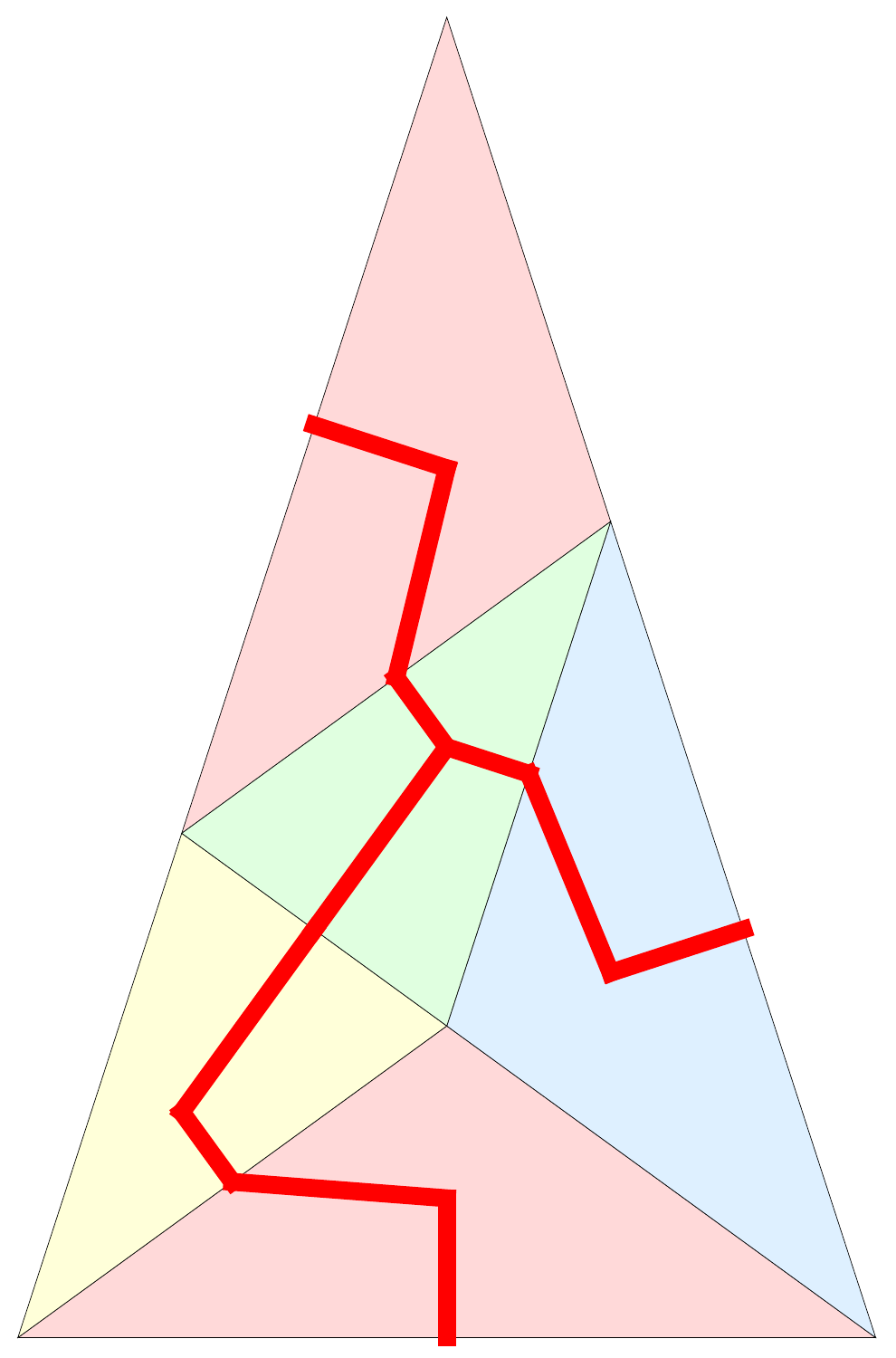}}
	edge[<-] (f);
\node (h) at (10,1.5) [label=below:{$A$}] {\includegraphics[width=1.5cm]{penrose_lim_3}}
	edge[<-,dashed] (g);
\end{tikzpicture}
\caption{A fractal dual for the Penrose prototiles} \label{eg:penrosefractal}
\end{figure}

Placing the graph $A$ in each tile of a Penrose tiling $T$, defines a new tiling whose tiles have fractal borders. Figure \ref{eg:penrosefractaltiling} shows a patch of the fractal dual tiling of the Penrose tiling overlaid on a patch of the original Penrose tiling. The substitution on the fractal realisation is inherited from the original Penrose tiling. See \cite[Appendix A]{FWW} for several additional examples.

\begin{figure}[h!]
\centering
\frame{\includegraphics[trim=0 25 0 35,clip,width=0.9\textwidth]{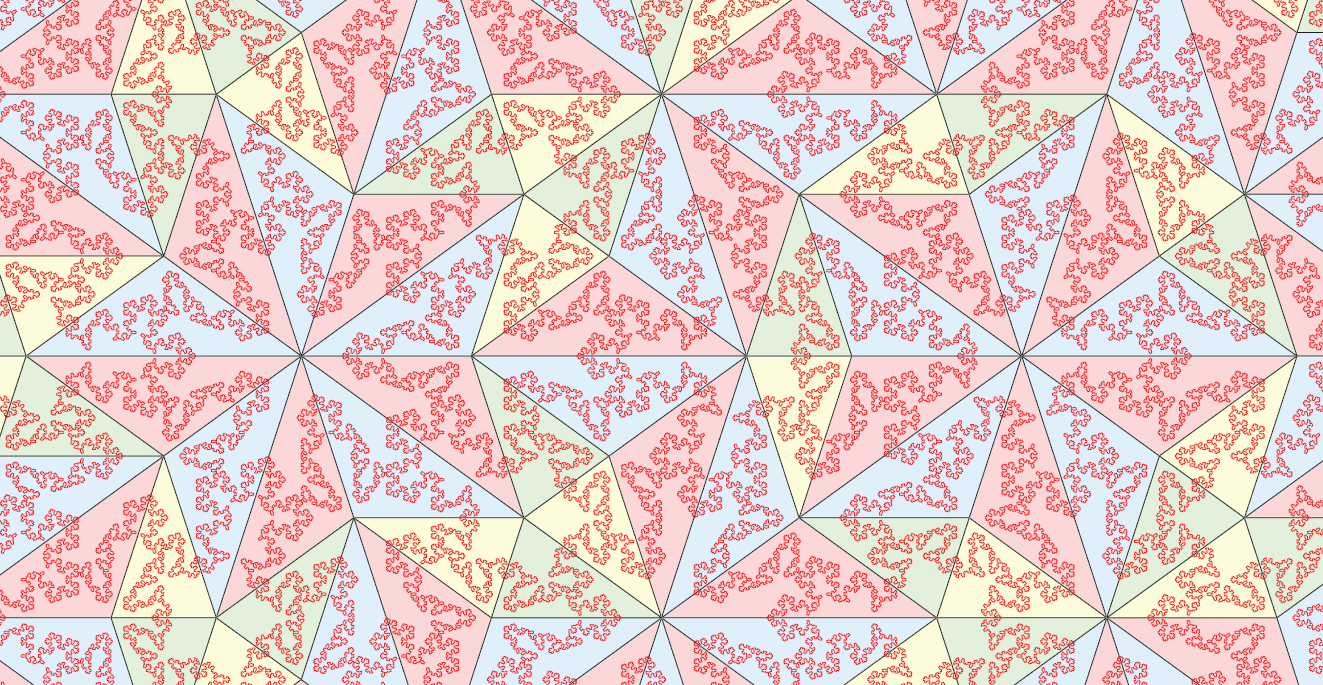}}
\caption{A fractal dual tiling} \label{eg:penrosefractaltiling}
\end{figure}

\end{example}

Now that we have defined recurrent pairs and the associated fractal dual, we are now able to define fractal trees. These are defined as decorations on the tiling $T_p$ for each $p \in \PP$. In short, a fractal tree on a tiling is a network of fractal edges, such that the punctures of two distinct tiles are uniquely connected by a fractal path. We use fractal trees to define spectral triples.

Let $(G,S)$ be a recurrent pair of quasi-dual trees, such that for each $p \in \PP$, the puncture $ x (p) $ is an interior vertex in $ G_p $. The algorithm in the proof of \cite[Theorem 6.6]{FWW} implies it is always possible to set the corresponding interior vertex in $S_p$ to be $x(p)$ as well. This forces an interior vertex of the fractal graph $A_p$ to be $x(p)$. We make this additional assumption for the remainder of this paper.

\begin{thm} \label{fractaltree}
Let $\omega$ be a primitive substitution rule on a set of prototiles $\PP$, with a recurrent pair $(G,S)$ of quasi-dual trees. Then, for each $p \in \PP$, there exists a geometric graph $\mathfrak{F}_p$ in $T_p$ such that $\mathfrak{F}_p$ is a connected tree, and $ x(t) \in \mathfrak{F}_p $ for each tile $t \in T_p$.
\end{thm}

\begin{proof}
Let $(G,S)$ be a recurrent pair of quasi-dual trees on $\PP$, and let $A$ be the associated fractal graph. We first define a subgraph $A^\circ$ of $A$ that does not meet the boundary of any prototile. For each $p \in \PP$, view $ A_p $ as a graph consisting of edges in $ \RR (A)_p $, and define
\[
A^\circ_p := A_p \setminus \{ e \in E (\RR(A)_p) : \text{an endpoint of} \ e \ \text{is in} \ \partial V (A_p) \},
\]
then $A^\circ := \sqcup_{p \in \PP}\ A^\circ_p$ is a subgraph of $A$ that does not meet the boundary of any prototile.

We now claim that we can construct a graph $F_1 \subset \RR(A)$ in $\PP$ such that for each $ p \in \PP $,
\begin{enumerate}[(1)]
\item $(F_1)_p$ is a connected tree; \label{connect}
\item $ x (t) \in V ((F_1)_p) $ for all $ t \in \RR (\PP)_p $; and \label{puncture}
\item $(F_1)_p \subset \interior(\supp(p))$ with $A^\circ_p$ a subgraph of $(F_1)_p$. \label{interior}
\end{enumerate}
Before constructing the graph $F_1$, we comment on the significance of conditions \eqref{connect}--\eqref{interior}. Condition \eqref{connect} ensures that any pair of vertices in $(F_1)_p$ are connected by a unique path for each $p \in \PP$. Condition \eqref{puncture} says that the graph $(F_1)_p$ passes through the puncture of every subtile in $\RR(\PP)_p$ for each $p \in \PP$. This means that every subtile is connected to the graph $F_1$ via its puncture. Condition \eqref{interior} ensures that the graph $(F_1)_p$ is wholly contained within the interior of the support of $p$ for each $p \in \PP$. This means that the graph contains no edges which have a boundary vertex as one of its endpoints. Moreover, condition \eqref{interior} also says that we are retaining the basic structure of $ A $.

We now construct the graph $ F_1 $ inductively. First, set $ B_1 := A^\circ $. For $ k \in \N $, if $ B_k $ does not satisfy conditions \eqref{connect}--\eqref{interior}, then there exists $ p \in \PP $, and a subtile $ t \in \RR (\PP)_p $ such that $ x (t) \notin V ((B_n)_p) $. We construct $ B_{k+1} $ as follows. For each $ q \in \PP $ such that $ q \neq p $, we define $ (B_{k+1})_q := (B_k)_q $. Now, since $ \RR (A)_p $ connects the punctures of the subtiles in $ \RR(\PP)_p $, there is a path $ \mu $ in $ \RR (A)_p $ connecting $ x (t) $ to $ (B_k)_p $ which intersects $ (B_k)_p $ at a single vertex. Define $ (B_{k+1})_p $ to be the graph $ (B_k)_p $ along with the edges comprising the path $ \mu $. Then $ x (t) $ is connected to the graph $ (B_{k+1})_p $ as a degree one vertex. Since there are only a finite number of prototiles, and hence subtiles, there exists $N \in \N$ sufficiently large such that $x(t) \in B_N$ for all subtitles $t \in \RR(\PP)$. Then it is routine to check that the graph $F_1:=B_N$ satisfies \eqref{connect}--\eqref{interior}, proving the claim.

For $ n > 2 $, we inductively define $F_n$ by the formula,
\begin{equation} \label{eq:induct_graph}
F_n := \RR(F_{n-1}) \cup F_{n-1}.
\end{equation}
We claim that for all $n \in \N$, $F_n$ is a graph in $ \PP $ such that for each $ p \in \PP $,
\begin{enumerate}[(i)]
\item $(F_n)_p$ is a connected tree; \label{ind_connect}
\item $ x (t) \in V ((F_n)_p) $ for all $ t \in \RR^n (\PP)_p $; and \label{ind_puncture}
\item $(F_n)_p \subset \interior(\supp(p))$. \label{ind_interior}
\end{enumerate}
Since $ F_n $ is built from the attractor $ A $, the formula for $ F_n $, given in equation \eqref{eq:induct_graph}, gives a well-defined subgraph of $ \RR^n (A) $. For the remainder of the claim, we proceed by induction on $n$. By conditions \eqref{connect}--\eqref{interior}, $F_1$ satisfies \eqref{ind_connect}--\eqref{ind_interior}. Assume $F_{n-1}$ satisfies \eqref{ind_connect}--\eqref{ind_interior}. We prove $F_n$ does as well. 
\begin{enumerate}[(i)]
\item By our inductive hypothesis, $\RR(F_{n-1})_p$ is a connected tree in each subtile $t \in \RR(\PP)_p$. Moreover, there exists a unique path between any two subtiles in $ \RR (\PP)_p $ via the connected tree $ (F_{n-1})_p $. Thus, the union $\RR(F_{n-1})_p \cup (F_{n-1})_p$, which is a subgraph of $ \RR^n (A) $, is a connected tree in $\RR^n(\PP)_p$. We remark that the invariance of $A$ under the map $\RR$ is essential for this step to work, which is why we require fractal trees rather than trees.
\item By our inductive hypothesis, $ x (t) \in V ((F_{n-1})_p) $ for all $ t \in \RR^{n-1}(\PP)_p$. Thus, $ x (t) \in V (\RR (F_{n-1})_p) $ for all $ t \in \RR^n (\PP)_p $. Since $\RR(F_{n-1})_p \subset (F_n)_p$, we have $ x (t) \in V ((F_n)_p) $ for all $ t \in \RR^n (\PP)_p $.
\item By our inductive hypothesis, $(F_{n-1})_p \subset \interior(\supp(p))$ and hence, $\RR(F_{n-1})_p$ is contained in the union of the interior of the subtiles in $\RR(\PP)_p$. That is, $\RR(F_{n-1})_p \subset \interior(\supp(p))$.
\end{enumerate}
Thus by induction, there exists a geometric graph $(F_n)_p$ satisfying conditions \eqref{ind_connect}--\eqref{ind_interior} for each $n \in \N$.

We now define a fractal graph $\mathfrak{F}_n$ on $\omega^n(\PP)$ for each $n \in \N$ in the following way. For each $p \in \PP$ and $n \in \N$, define $(\mathfrak{F}_n)_p := \lambda^n(F_n)_p$. By construction, $(\mathfrak{F}_n)_p$ is a geometric graph in $\omega^n(p)$ containing the puncture of each tile in $\omega^n(p)$. Moreover, $(\mathfrak{F}_n)_p \subset (\mathfrak{F}_{n+1})_p$ for each $n \in \N$.
By construction of $T_p$, we have $\omega^n(p) \subset \omega^{n+1}(p)$ for all $ n \in \N $, so that
\[
\mathfrak{F}_p := \bigcup_{n=1}^{\infty}\ (\mathfrak{F}_n)_p,
\]
is a geometric graph on $T_p$ connecting the punctures in $T_p$ by a unique fractal path.
\end{proof}

\begin{definition}
Let $ \omega $ be a primitive substitution rule on a set of prototiles $ \PP $, with a recurrent pair $ (G,S) $ or quasi-dual trees. For each $p \in \PP$, we will refer to a geometric graph $\mathfrak{F}_p$, as in the statement of Theorem \ref{fractaltree}, as a \emph{fractal tree} on $T_p$.
\end{definition}

For fixed $p \in \PP$, suppose that $T_p$ is endowed with a fractal tree $\mathfrak{F}_p$. We define a distance between two tiles in $T_p$ via $\mathfrak{F}_p$. Since fractal edges typically have infinite Euclidean length, the length of a fractal path is defined using Perron-Frobenius theory. Let $(G,S)$ be the recurrent pair which we used to construct the fractal tree $\mathfrak{F}_p$, as in Theorem \ref{fractaltree}, and let $A$ be the corresponding fractal graph. Since $\mathfrak{F}_p$ consists of edges in $A$, we assign a length to each edge in $A$ and then add the appropriate edge lengths to obtain a length between punctures of tiles in $T_p$.

Since $(G,S)$ is a recurrent pair, $G$ and $S$ are homeomorphic geometric graphs (planar graphs embedded in $\R^2$). Let $\psi : G \to S$ denote the homeomorphism. For each edge $e \in E(G)$, the edge $\psi(e) \in E(S)$ is composed of edges of the form $\lambda \inverse(f+x)$ where $f \in E(G)$, and $x \in \R^2$. We may view $\psi$ as a substitution rule on the set of edges.

\begin{definition}
Let $ \PP $ be a set of prototiles, and $ (G,S) $ a recurrent pair in $ \PP $. Denote by $ \psi $ the homeomorphism between $ G $ and $ S $. For each $p,q \in \PP$, $e \in G_p$ and $f \in G_q$, define
\[
D^E_{ef} := \{ x \in D_{pq} : \lambda \inverse (f+x) \subseteq \psi(e) \}.
\]
The matrix $D^E := (D^E_{ef})$ is called the \emph{edge digit matrix}. Further, let $M^E_{ef} := |D^E_{ef}|$. That is, $M^E_{ef}$ is the number of scaled copies of $f$ in $\psi(e)$. The matrix $M^E := (M^E_{ef})$ is the \emph{edge substitution matrix}.
\end{definition}

When $M^E$ is primitive, the Perron-Frobenius Theorem defines a unique real eigenvalue $\kappa >1$, with associated eigenvector $v = (v_e)_{e \in E(G)}$, whose entries are all positive real numbers. We note that $M^E$ may not be primitive for all choices of recurrent pairs, however, it is routine to see that a recurrent pair $(G,S)$ with primitive substitution matrix $M^E$ can always be chosen for all tilings satisfying the hypothesis of \cite[Theorem 6.6]{FWW} (see also Theorem \ref{thm6.6-FWW}). 

\begin{definition}
Suppose that $ (G,S) $ is a recurrent pair in a set of prototiles $ \PP $, and let $ A $ be the associated fractal graph. Let $v = (v_e)_{e \in E(G)}$ be the Perron-Frobenius eigenvector for the primitive edge substitution matrix $M^E$. For each $e \in E(A)$, we write $ \ell (e) := v_e $ for the \emph{length} of $e$.
\end{definition}

\begin{definition} \label{defn:fractal_distance}
Let $ \PP $ be a set of prototiles, and fix $ p \in \PP $. Suppose $t,t' \in T_p$, and $\mathfrak{F}_p$ is a fractal tree on $T_p$. The \emph{fractal distance} between $t$ and $t'$, denoted $d_{\mathfrak{F}_p}(t,t')$, is the sum of the fractal edges making up the unique fractal path in $\mathfrak{F}_p$ between the punctures $x(t)$ and $x(t')$.
\end{definition}

\begin{example}[A fractal tree for the Penrose tiling] \label{eg:fractaltree}

In this example, we illustrate the construction of a fractal tree for the recurrent pair defined on the Penrose tiling in Example \ref{penrose_recurrentpair}. We then use the Perron-Frobenius Theorem to assign lengths to the fractal edges depicted in Figure \ref{fig:fractal_edges}, which extend to the edges in all forty prototiles by rotation.
Putting these two steps together defines a fractal distance between the punctures of tiles $t$ and $t'$ in $T_p$, for all $p \in \PP$. 

\begin{center}
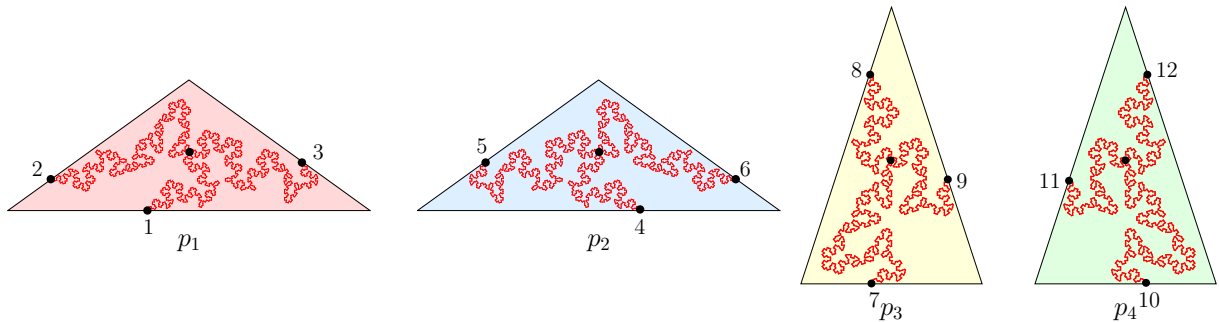
\begin{figure}[H]
\scalebox{0.7}{
\begin{tikzpicture}[>=stealth,scale=1.1]
\node at (0,0) [label=below:{\large $p_1$}] {\includegraphics[width=0.4\textwidth]{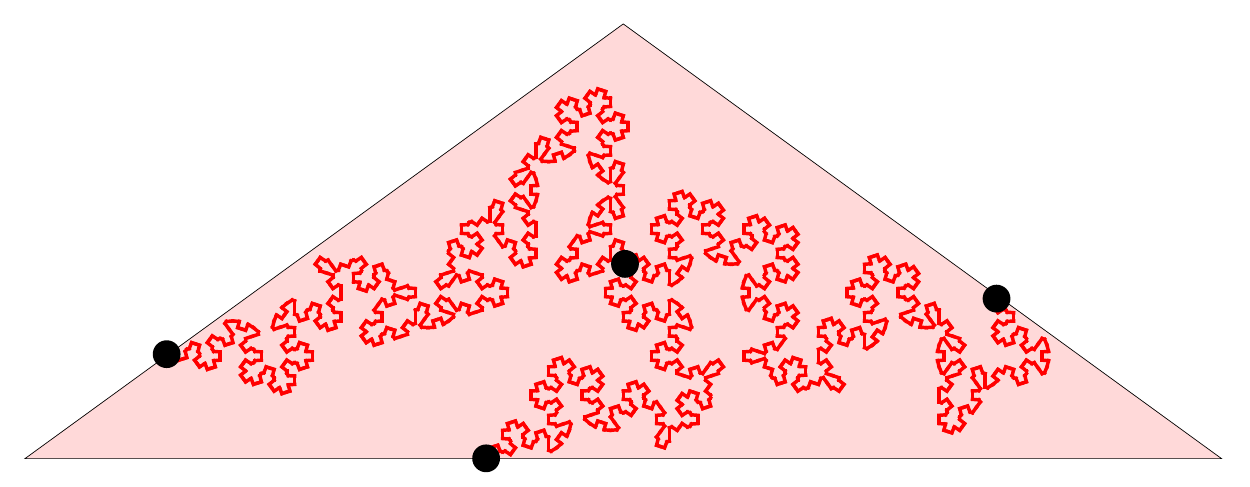}};
\node at (7,0) [label=below:{\large $p_2$}] {\includegraphics[width=0.4\textwidth]{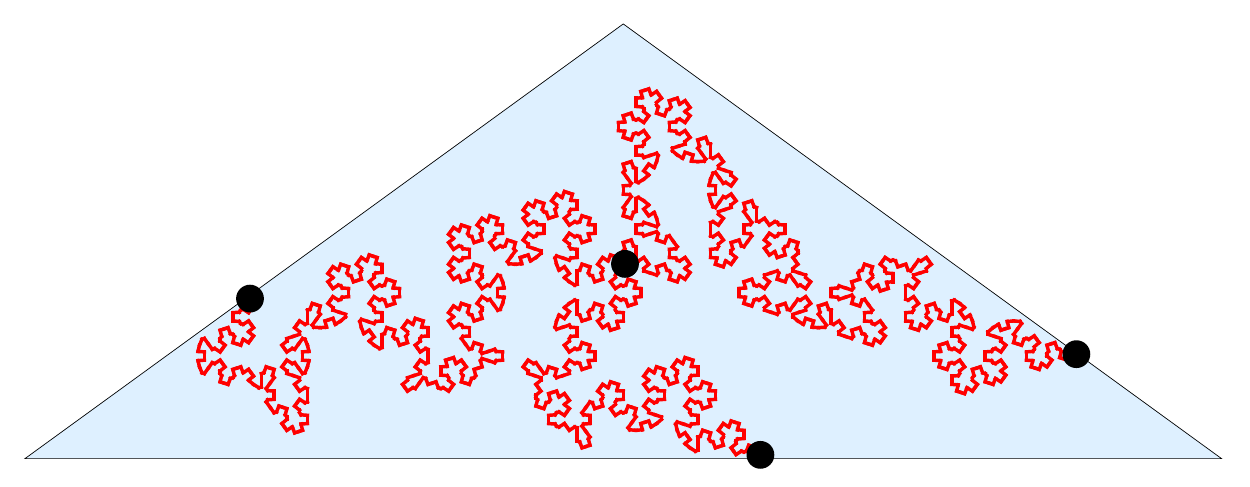}};
\node at (12,0) [label=below:{\large $p_3$}] {\includegraphics[width=0.2\textwidth]{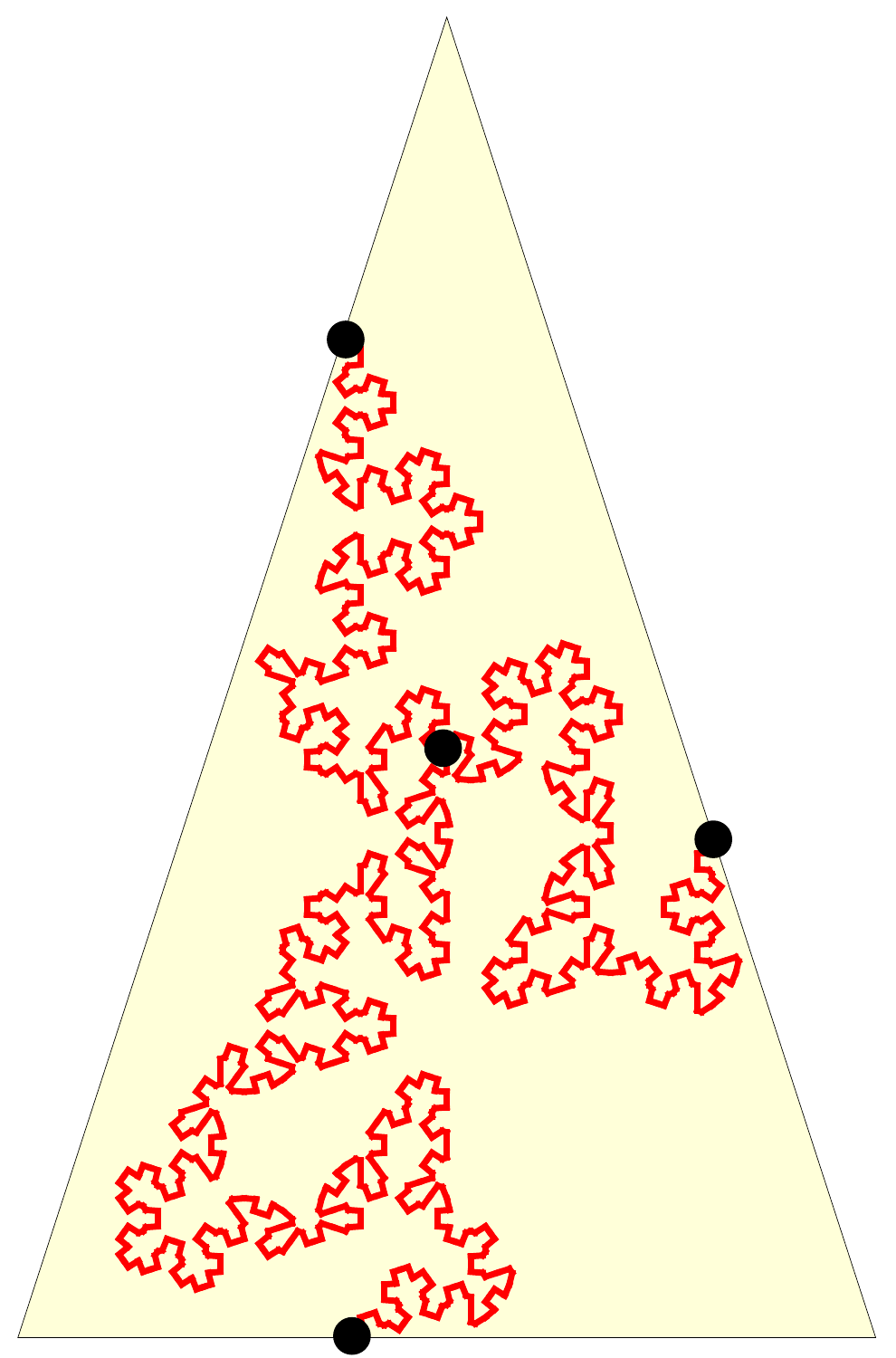}};
\node at (16,0) [label=below:{\large $p_4$}] {\includegraphics[width=0.2\textwidth]{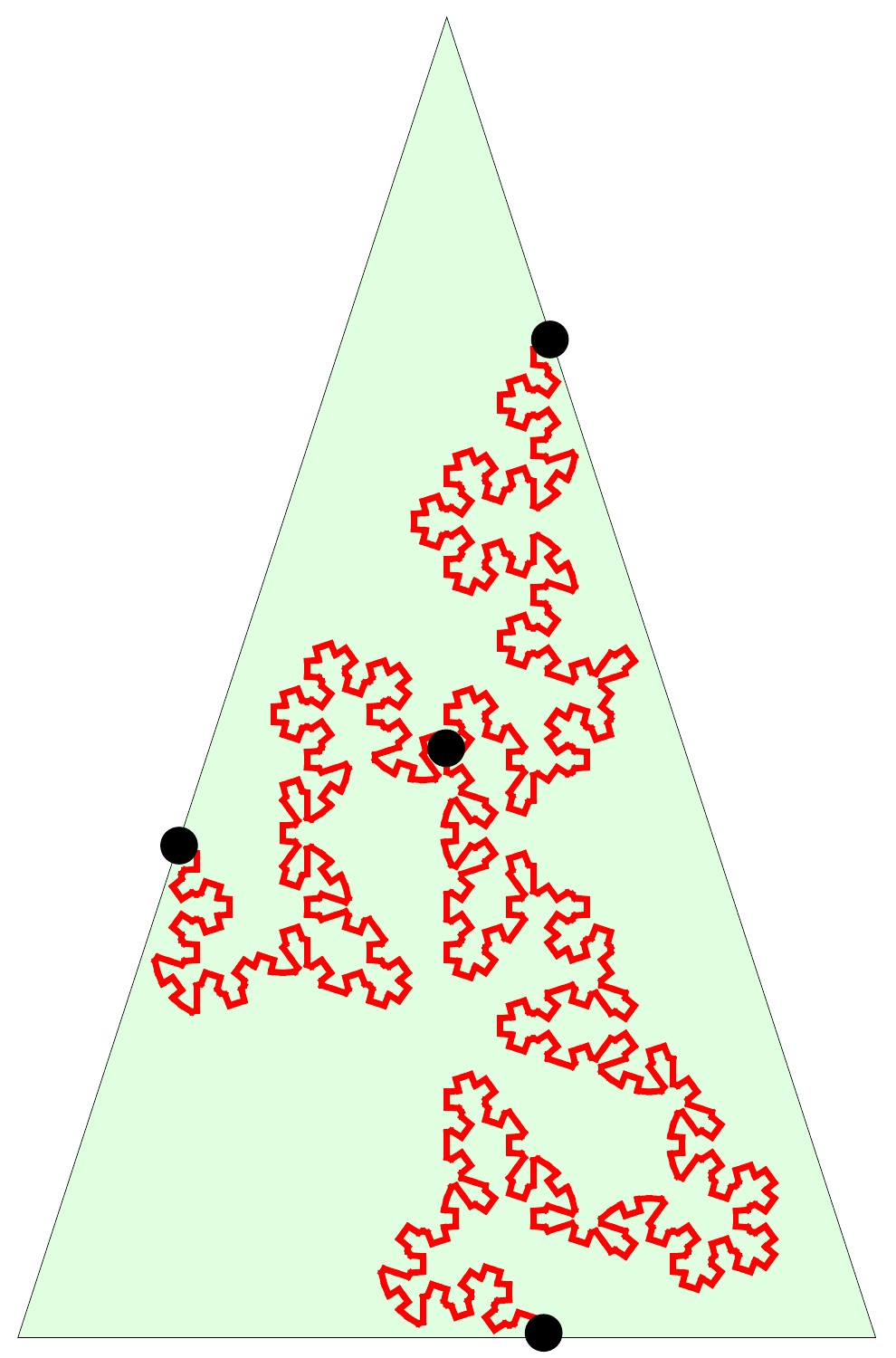}};
\node at (-0.7,-1.4) {1};
\node at (-2.6,-0.4) {2};
\node at (2.2,-0.1) {3};
\node at (7.7,-1.4) {4};
\node at (5,0) {5};
\node at (9.5,-0.4) {6};
\node at (11.7,-2.65) {7};
\node at (11.4,1.3) {8};
\node at (13.2,-0.6) {9};
\node at (16.4,-2.65) {10};
\node at (14.7,-0.6) {11};
\node at (16.7,1.3) {12};
\end{tikzpicture}}
\caption{12 of the fractal edges in the graph $A$ for the Penrose tiling}
\label{fig:fractal_edges}
\end{figure}
\end{center}

To elucidate the algorithm used in the proof of Theorem \ref{fractaltree}, we restrict our attention to prototile $p_1$. Starting with the recurrent pair in Example \ref{penrose_recurrentpair}, the first step is to construct the graph $A^\circ_{p_1}$ by removing the edges of $\RR(A)_{p_1}$ in $ A_{p_1} $ that intersect the boundary of $ p_1 $. From here, we need to add in extra edges from $ \RR (A)_{p_1} $ to form a graph $ (F_1)_{p_1} $ in $ p_1 $ which passes through the puncture of each subtile in $ \RR (\PP)_p $. Note that in this example, we do not need to add in any extra edges since all punctures are already contained in $ A^\circ_p $. The graph $(F_1)_{p_1}$ appears in Figure \ref{fig:fractal_tree1}.

\begin{center}
\begin{figure}[H]
\scalebox{0.9}{
\begin{tikzpicture}[>=stealth,scale=1.1]
\node at (0,0) {\includegraphics[width=0.8\textwidth]{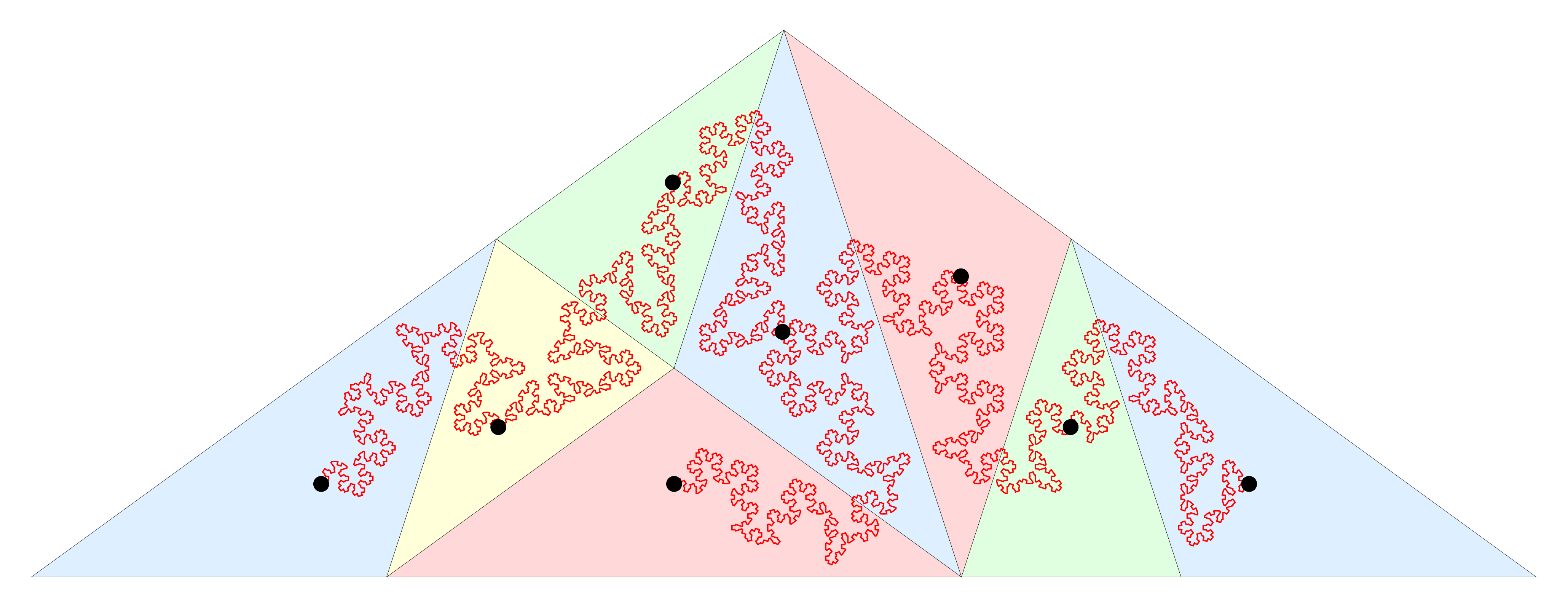}};
\end{tikzpicture}}
\caption{$(F_1)_{p_1}$ - prototile $p_1$ after removing the edges of $\RR(A)_{p_1}$ that intersect the boundary}
\label{fig:fractal_tree1}
\end{figure}
\end{center}

The second step in the algorithm, illustrated in Figure \ref{fig:fractal_tree2}, is to consider the graph $\RR(F_1)$, which will define a graph in each prototile that intersects each puncture in $\RR^2(\PP)$.
Notice, however, that for each $ q \in \PP $, $ (F_1)_q $ contains no edges whose endpoints are boundary vertices, and hence, not all pairs of punctures of the $ 2 $-subtiles in $ \RR^2 (\PP)_{p_1} $ are connected via the graph $ \RR (F_1)_{p_1} $.

\begin{center}
\begin{figure}[H]
\scalebox{0.9}{
\begin{tikzpicture}[>=stealth,scale=1.1]
\node at (0,0) {\includegraphics[width=0.8\textwidth]{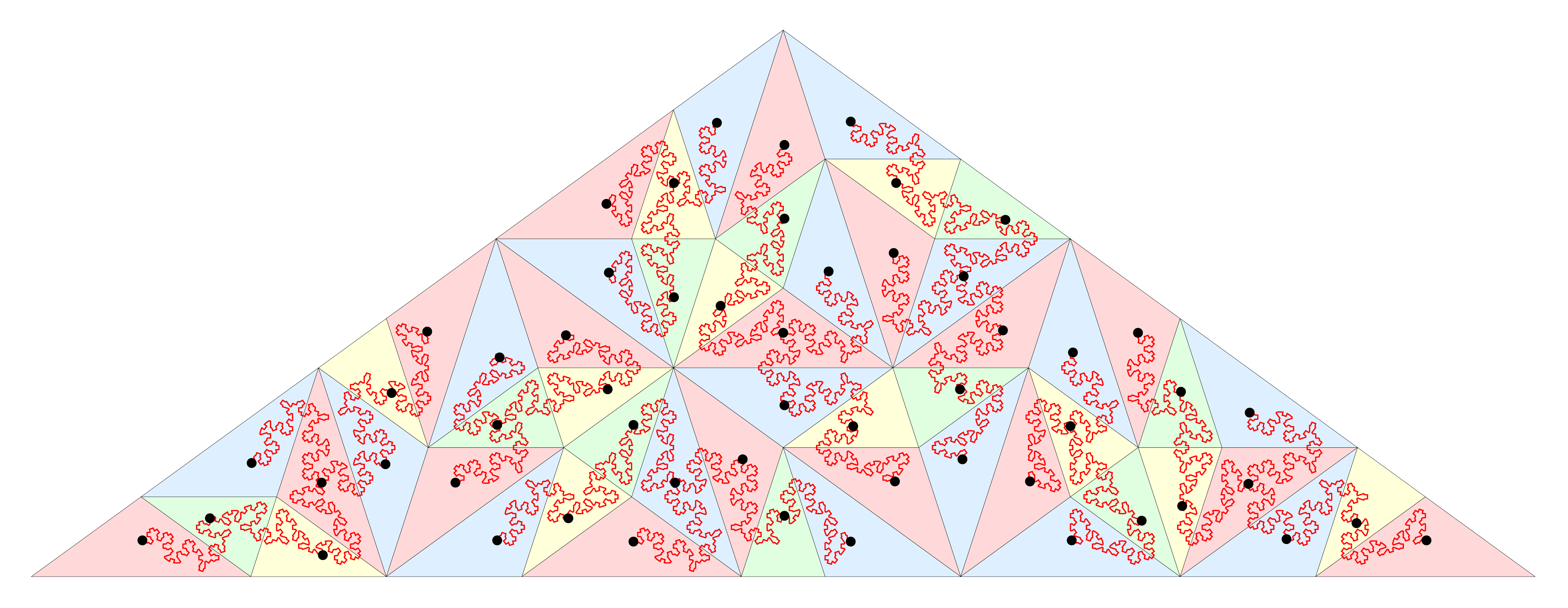}};
\end{tikzpicture}}
\caption{$\RR(F_1)_{p_1}$ - prototile $p_1$ after applying $\RR$ to the graph $F_1$}
\label{fig:fractal_tree2}
\end{figure}
\end{center}

We define $ F_2 := \RR (F_1) \cup F_1 $, so that each pair of connected graphs in $ F_1 $ are joined by additional fractal edges. We note that it is essential that we begin with a fractal graph, so that $F_1$ and $\RR(F_1)$ intersect along complete edges. In Figure \ref{fig:fractal_tree3}, we illustrate the graph $ (F_2)_{p_1} $.

\begin{center}
	\begin{figure}[H]
		\scalebox{0.9}{
			\begin{tikzpicture}[>=stealth,scale=1.1]
			\node at (0,-14) {\includegraphics[width=0.8\textwidth]{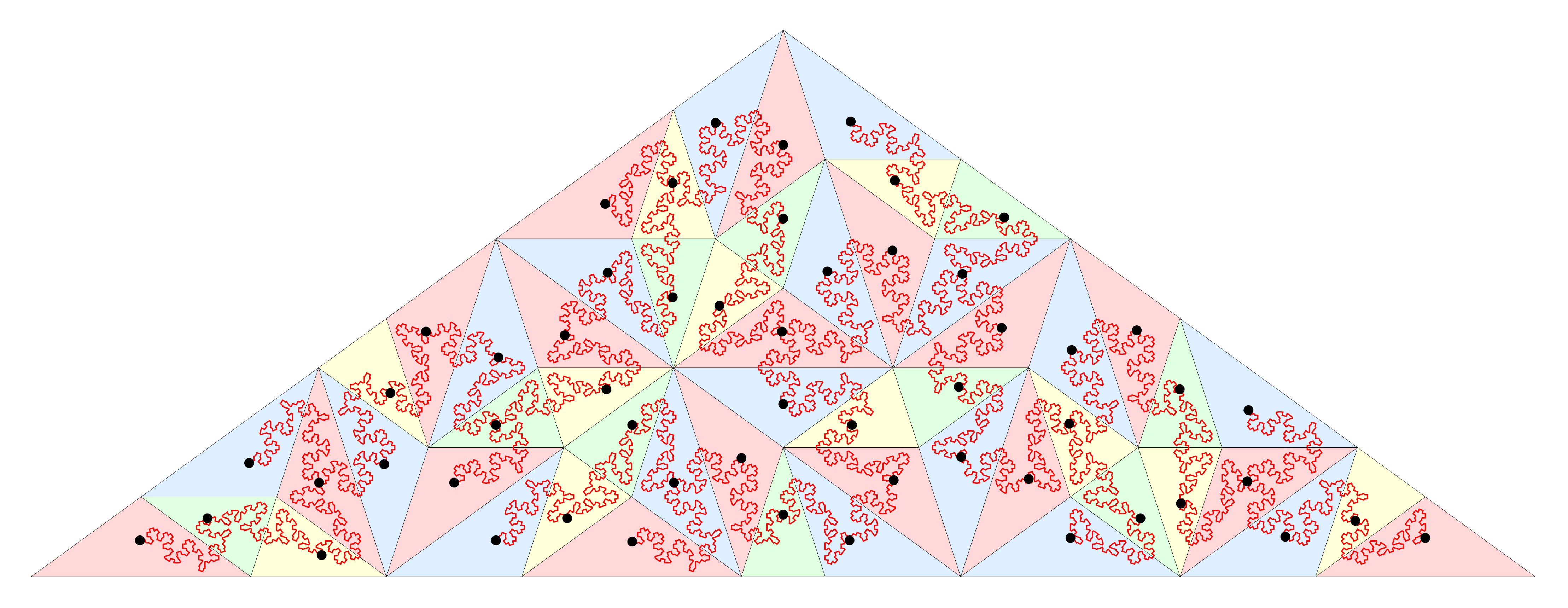}};
			%\draw[->] (8,-5.5) -- (9,-4.5);
			\end{tikzpicture}}
		\caption{$(F_2)_{p_1} = (\RR(F_1) \cup F_1)_{p_1}$ - the graph $F_2$ restricted to prototile $p_1$}
		\label{fig:fractal_tree3}
	\end{figure}
\end{center}

Inductively, we obtain graphs $F_n$ in $ \PP $ for all $n \in \N$. For each $n \in \N$, consider $\mathfrak{F}_n=\lambda^{n}F_n$ which connects the punctures of each tile in each patch $\omega^n(p)$ for each $p \in \PP$. Finally, the infinite graph $\mathfrak{F}_{p_1}$ on tiling $T_{p_1}$, is defined to be the infinite union $\cup_{n=1}^\infty (\mathfrak{F}_n)_{p_1}$. Thus we have a fractal tree on the tiling $T_{p_1}$ that jointly extends each of the fractal trees $(\mathfrak{F}_n)_{p_1}$, and the patches $\omega^n(p_1)$.

We now assign lengths to the edges $1$--$12$ in Figure \ref{fig:fractal_edges}. The matrix $M^E$ for the recurrent pair $(G,S)$ defined in Example \ref{penrose_recurrentpair}, is a primitive $120 \times 120$ matrix. However, since the edges in Figure \ref{fig:fractal_edges} extend to the $10$-fold rotations of the $4$ prototiles, we need only consider a $12 \times 12$ matrix. The Perron-Frobenius eigenvalue is $\kappa=\frac{1}{2} (3+\sqrt{21+4\sqrt{21}}) \approx 4.6357$ with associated eigenvector 
\[
v=[0.6952, \,\,\, 1.3953, \,\,\, 1.2638, \,\,\, 0.6952, \,\,\, 1.2638, \,\,\, 1.3953, \,\,\, 1, \,\,\, 0.5452, \,\,\, 0.5686, \,\,\, 1, \,\,\, 0.5686, \,\,\, 0.5452 ]^T.
\]
For each $ 1 \leq i \leq 12 $, entry $v_i$ is the length $l(i)$ of edge $i$. To obtain the fractal distance $d_{\mathfrak{F}}(t,t')$ between two tiles $t,t' \in T_p$, we sum the lengths of the edges between the punctures $x(t)$ and $x(t')$. Thus, the reader can verify that the fractal distances between the middle pink tile in Figure \ref{fig:fractal_tree3}, and each of its three adjacent tiles is $l(2)+l(8)=1.9405$, $l(3)+l(5)=2.5275$, and $l(1)+l(4)=1.3904$ in the northwest, northeast, and south directions, respectively. 
\end{example}

For the remainder of this section, fix a prototile set $ \PP $ with a substitution rule $ \omega $.
Furthermore, for each $ p \in \PP $, let $ \mathfrak{F}_p $ be a fractal tree in $ T_p $. Write $ \kappa > 1 $ for the Perron-Frobenius eigenvalue of the associated edge substitution matrix. We explore some properties associated with the fractal tree, and make some new definitions which will aid us in the definition of a spectral triple.

\begin{definition} \label{root_of_prototile}
Let $ p \in \PP $, $ x \in \R^2 $, and $ n \in \N $. The \emph{root} of the patch $ \omega^n (p) + x $ is the tile $ p+x $. In particular, taking $ x = 0 \in \R^2 $, $p$ is the root of the patch $\omega^n(p)$ for each $ p \in \PP $, and $n \in \N$.
\end{definition}

\begin{lemma} \label{fscale}
Fix $ p \in \PP $, and suppose that $t,t',r,r' \in T_p$ are tiles in $ T_p $ such that $r$ is the root of $\omega(t)$, and $r'$ is the root of $\omega(t')$. Then, $d_{\mathfrak{F}_p}(r,r') = \kappa d_{\mathfrak{F}_p}(t,t').$
\end{lemma}

\begin{proof}
Choose $N$ sufficiently large such that $\lambda^{-N} t,\lambda^{-N} t' \in \RR^N(\PP)_p$ as $ N $-subtiles. By construction, there is a unique fractal path between $ x (\lambda^{-N} t) $ and $ x (\lambda^{-N} t') $ via the tree $(F_N)_p$, as in the proof of Theorem \ref{fractaltree}. By equation \eqref{eq:induct_graph}, we have $ (F_{N+1})_p = (\RR(F_N) \cup (F_N))_p $ so that $(F_N)_p \subset (F_{N+1})_p$. Thus, the fractal path between $\lambda^{-(N+1)} r$ and $\lambda^{-(N+1)} r'$ in $\RR^{N+1}(\PP)_p$, is identical to the fractal path between $\lambda^{-N} t$ and $\lambda^{-N} t'$ in $\RR^N(\PP)_p$. Scaling appropriately, the fractal path between $r$ and $r'$ is the image under the edge digit matrix of the fractal path between $t$ and $t'$. Since the length of fractal edges were assigned by the eigenvector corresponding to the edge substitution matrix with Perron-Frobenius eigenvalue $ \kappa > 1 $, we conclude that $d_{\mathfrak{F}_p}(r,r') = \kappa d_{\mathfrak{F}_p}(t,t')$.
\end{proof}

\begin{definition} \label{radii}
For each $p \in \PP$, and $n \in \N$, the \emph{$n$th major branch radius of $p$} is the number
\[
\LL_n(p) := \max\{d_{\mathfrak{F}_p}(t,p) : t \in \omega^n(p) \setminus \{p\} \}.
\]
A tile $t \in \omega^n(p)$ satisfying $d_{\mathfrak{F}_p}(t,p) = \LL_n(p)$ is called an \emph{$n$th major leaf of $p$}. The \emph{major branch radius} is $\LL := \max\{\LL_1(p) : p \in \PP \}$. Similarly, for each $p \in \PP$ and $m,n \in \N$ such that $m > n$ the \emph{$m$th minor branch radius of $p$ over $n$} is the number
\[
\SS_{m,n}(p) = \min\{d_{\mathfrak{F}_p}(t,p) : t \in \omega^m(p) \setminus \omega^n(p) \}.
\]
If $n = 0$, we have $ \omega^0 (p) = \{p\} $, and we write $\SS_m(p) := \SS_{m,n}(p)$ for the \emph{$m$th minor branch radius of $p$}. A tile $t \in \omega^m(p) \setminus \omega^n(p)$ satisfying $d_{\mathfrak{F}_p}(t,p) = \SS_{m,n}(p)$ is called an \emph{$m$th minor leaf of $p$ over $n$}. In particular, a tile $t \in \omega^m(p) \setminus \{p\}$ satisfying $d_{\mathfrak{F}_p}(t,p) = \SS_m(p)$ is called an \emph{$m$th minor leaf of $p$}. The \emph{minor branch radius} is given by $\SS := \min\{\SS_n(p) : p \in \PP, n \in \N \}$.
\end{definition}

\begin{lemma} \label{longest}
For each $p \in \PP$, and $n \in \N$, we have $\LL_n(p) \leq \sum_{i=0}^{n-1} \kappa^i \LL$.
\end{lemma}

\begin{proof}
We proceed by induction. By defintion, we have $\LL = \{\LL_1(p) : p \in \PP \}$, and hence, $\LL_1(p) \leq \LL$ for all $p \in \PP$. Assume for some $ k \in \N $, we have $ \LL_k (p) \leq \sum_{i = 0}^{k-1} \kappa^i \LL $ for all $ p \in \PP $. Fix $p \in \PP$, and suppose $t$ is a first major leaf of $p$. Let $r$ be the root of $\omega^k(t)$. Since $p$ is the root of $\omega^k(p)$,
\[
d_{\mathfrak{F}_p}(r,p) = \kappa^k d_{\mathfrak{F}_p}(t,p) = \kappa^k \LL_1(p),
\]
by Lemma \ref{fscale}. Thus, if we view the patch $\omega^{k+1}(p)$ as a collection of supertiles of order $k$, the fractal distance from $p$ to the root of one of the supertiles of order $k$ is at most $\kappa^k \LL_1(p)$. Let $t'$ be a $(k+1)$th major leaf of $p$. Then $t' \in \omega^k(t'')$ for some $t'' \in \omega(p) \setminus \{p\}$. That is, $t'$ must be contained in one of the supertiles of order $k$. Let $r'$ be the root of $\omega^k(t'')$. Then, by our inductive hypothesis,
\[
d_{\mathfrak{F}_p}(t',p) \leq d_{\mathfrak{F}_p}(t',r') + d_{\mathfrak{F}_p}(r',p) \leq \sum_{i=1}^{k-1} \kappa^i \LL + \kappa^k \LL = \sum_{i=1}^{k} \kappa^i \LL. \qedhere
\]
\end{proof}

\begin{lemma} \label{shortcorona}
For each $p \in \PP$, and $m,n \in \N$ with $m > n$, we have $\SS_{m,n}(p) \geq \kappa^{n-1} \SS$.
\end{lemma}

\begin{proof}
Let $p \in \PP$, and $m \in \N$. Suppose that $n = 1$. Then
\[
\SS_{m,1}(p) = \min\{d_{\mathfrak{F}_p}(t,p) : t \in \omega^m(p) \setminus \omega(p) \}  \geq \min\{d_{\mathfrak{F}_p}(t,p) : t \in \omega^m(p) \setminus \{p\} \} = \SS_m(p) \geq \SS.
\]
Suppose that $n > 1$. Let $t$ be a first minor leaf of $p$. Let $r$ be the root of $\omega^{n-1}(t)$, and note that $r \in \omega^n(p)$. Moreover, $ r $ satisfies $d_{\mathfrak{F}_p}(r,p) = \kappa^{n-1} \SS_1(p) \geq \kappa^{n-1} \SS$. Let $t'$ be an $m$th minor leaf of $p$ over $n$, then the fractal path between $p$ and $t'$ must pass through the root of $\omega^{n-1}(t'')$ for some $t'' \in \omega(p) \setminus \{p\}$. Thus, $ \SS_{m,n}(p) = d_{\mathfrak{F}_p}(t',p) \geq d_{\mathfrak{F}_p}(r,p) \geq \kappa^{n-1} \SS $.
\end{proof}

Given a set of prototiles $ \PP $ with a substitution rule $ \omega $, and fractal trees on $ T_p $ for each $ p \in \PP $, we now have a means of estimating the various major and minor branch radii. We will use these estimates later in the proofs of our results in Section \ref{sec:spectraltriples}, where we define spectral triples.

\section{$C^*$-algebras associated to nonperiodic tilings}

In this section, we investigate the $C^*$-algebra associated to a nonperiodic substitution tiling given by Kellendonk in \cite{Kel1}, on which we define a spectral triple. Let $\omega$ be a primitive substitution rule on a set of prototiles $\PP$, and $\phull$ be the associated discrete hull. Recall the translational equivalence relation on the discrete hull, given in Section \ref{sec:discretehull}, which may be written as:
\[
R_{punc} := \{ (T,T') \in \phull \times \phull : T' = T-x(t) \text{ for some }\ t \in T \}.
\]
We endow $R_{punc}$ with the metric given by,
\[
d_R((T,T-x(t)),(T',T'-x(t'))) := d(T,T') + |x(t)-x(t')|,
\]
where $T,T' \in \phull$, $t \in T$, and $t' \in T'$. In this topology, $R_{punc}$ is an \'etale equivalence relation. Note that the topology here is not the same as the topology inherited from the product topology on $\phull \times \phull$. In particular, in the latter topology, $R_{punc}$ is not \'etale.

Following the construction of groupoid $C^*$-algebras given by Renault in \cite{Ren}, Kellendonk constructed a $C^*$-algebra in \cite{Kel1}, which we denote by $A_{punc}$. We now outline this construction. Consider the complex vector space $C_c(R_{punc})$; the continuous functions of compact support on $R_{punc}$. Endow $C_c(R_{punc})$ with multiplication and involution given by,
\begin{gather*}
(f \cdot g)(T,T') := \sum_{T'' \in [T]} f(T,T'') g(T'',T'), \ \text{and} \\
f^*(T,T') := \overline{f(T',T)},
\end{gather*}
respectively, where $f,g \in C_c(R_{punc})$, and $(T,T') \in R_{punc}$.
Under these operations, $C_c(R_{punc})$ is a $^*$-algebra. For each $T \in \phull$, we view $ T $ as a countable collection of tiles, and consider the \emph{induced representation from the unit space}, $\pi_T : C_c(R_{punc}) \to \BB(\ell^2(T))$, given in \cite[Section 3]{Whi_rot}. Explicitly,
\[
(\pi_T(f) \xi)(t) := \sum_{t' \in T} f(T-x(t),T-x(t')) \xi(t'),
\]
where $f\in C_c(R_{punc})$, $\xi \in \ell^2(T)$, and $t \in T$. It is well known that each induced representation $\pi_T$ is non-degenerate \cite[Proposition 1.7]{Ren}. Completing $C_c(R_{punc})$ in the \emph{reduced $C^*$-algebra norm},
\[
\|f\|_{red} := \sup\{ \|\pi_T(f)\| : T \in \phull \},
\]
where $f \in C_c(R_{punc})$, we obtain Kellondonk's $C^*$-algebra $A_{punc}$.

Given Kellendonk's $C^*$-algebra $A_{punc}$, we now describe a dense spanning set, see \cite[Section 2.2]{Kel1}. Let $P$ be a patch, and let $t,t'$ be tiles in $P$. Define a function $e(P,t,t') : R_{punc} \to \C$ by
\[
e(P,t,t')(T,T') := \begin{cases}
1 & \mbox{if}\ T \in U(P,t)\ \text{and}\ T' = T-x(t') \\
0 & \mbox{otherwise}
\end{cases}.
\]
Then, $e(P,t,t')$ is a partial isometry in $C_c(R_{punc})$. The linear span of all such functions $e(P,t,t')$ forms a dense spanning set for $A_{punc}$. We note some of the relations satisfied by these functions. Let $P_1$ and $P_2$ be patches, $t_1,t_1'$ be tiles in $P_1$, and $t_2,t_2'$ be tiles in $P_2$. It is easy to check that
\[
e(P_1,t_1,t_1')^* = e(P_1,t_1',t_1).
\]
Furthermore, if there exists a patch $P$ such that $P_1 \subset P$ and $P_2 \subset P$, then
\[
e(P_1,t_1,t_1') \cdot e(P_2,t_2,t_2') = e(P_1 \cup P_2,t_1,t_2').
\]
We note that $A_{punc}$ is a unital $C^*$-algebra, with unit
\[
1_{A_{punc}} = \sum_{p \in \PP} e(\{p\},p,p).
\]

To close this section, we define another representation of $ A_{punc} $ that we require in the next section. For each $p \in \PP$, consider the map $\pi_p : C_c(R_{punc}) \to \BB(\ell^2(T_p \setminus \{p\}))$ defined by
\begin{equation} \label{eq:repn}
(\pi_p(f) \xi)(t) := \sum_{t' \in T_p \setminus \{p\}} f(T_p-x(t),T_p-x(t')) \xi(t'),
\end{equation}
where $f \in C_c(R_{punc})$, $\xi \in \ell^2(T_p \setminus \{p\})$ and $t \in T_p \setminus \{p\}$. We note that $\pi_p$ is similar to an induced representation of the unit space. In fact, the only difference here is that we are removing the tile $p$ from the tiling $T_p$. The reason we have removed this particular tile will become apparent in the next section. That $\pi_p$ is a non-degenerate representation on $C_c(R_{punc})$ is similar to the proof for the induced representation found in \cite[Proposition 1.7]{Ren}. Since the substitution on $\Omega_{punc}$ is assumed to be primitive, every finite patch appearing in any tiling in $\Omega_{punc}$ also appears in $T_p \setminus\{p\}$. Using this fact, it is routine to show that $A_{punc}$ is isometrically isomorphic to the closure of $C_c(R_{punc})$ is the operator norm of $\BB(\ell^2(T_p \setminus \{p\}))$. Thus, $\pi_p$ extends to a faithful representation of $A_{punc}$.

\section{Spectral triples on nonperiodic tilings from fractal trees} \label{sec:spectraltriples}

In this section, we define a spectral triple on Kellendonk's $C^*$-algebra $A_{punc}$, introduced in the previous section. We begin with the definition of a spectral triple on a $C^*$-algebra.

\begin{definition} \label{spectral_triple}
Let $A$ be a unital $C^*$-algebra. A \emph{spectral triple over $A$}, $(A,H,D)$, consists of a separable Hilbert space $H$,  a faithful representation $\pi : A \to \BB(H)$, and an unbounded, self-adjoint operator $D$ on $H$, satisfying the following properties:
\begin{enumerate}[(1)]
\item\label{spectral triple 1} $\{ a \in A : \pi(a) \Dom(D) \subset \Dom(D)\ \text{and}\ [D,\pi(a)] \in \BB(H) \}$ is dense in $A$; and
\item\label{spectral triple 2} the operator $(1+D^2) \inverse$ is compact on $H$.
\end{enumerate}
A spectral triple $(A,H,D)$ is \emph{$\theta$-summable} if the operator $\exp(-\alpha D^2)$ is trace class for all $\alpha > 0$.
\end{definition}

We are now able to state the main result of the paper, which makes use of the fractal trees, and the fractal distance between tiles in a tiling, as defined in Section \ref{sec:fractal_trees}.

\begin{thm} \label{nposspectrip}
Let $ \PP $ be a set of prototiles with nonperiodic substitution rule $ \omega $. For each $p \in \PP$, let $T_p$ be the self-similar tiling from Definition \ref{tilinggeneratedbyp}. Let $(G,S)$ be a recurrent pair on $\PP$ satisfying the hypotheses of Theorem \ref{fractaltree}, so that there is a fractal tree $\mathfrak{F}_p$ on $T_p$, for all $p \in \PP$.
\begin{enumerate}
\item For each $p \in \PP$, there is a $\theta$-summable spectral triple $(A_{punc},H_p,D_p)$ where $H_p:= \ell^2(T_p \setminus \{p\})$, $\pi_p:A_{punc} \to \BB(H_p)$ is defined in equation \eqref{eq:repn}, and $D_p$ is an unbounded self-adjoint operator on $H_p$ defined on the canonical basis $\{\delta_t \mid t \in T_p \setminus \{p\}\}$ by $D_p \delta_t:=\ln\big(d_{\mathfrak{F}_p}(t,p)\big)\delta_t$. \label{spec_trip_p}
\item For each $\sigma : \PP \to \{-1,1\}$ there is a $\theta$-summable spectral triple $(A_{punc},H,D_{\sigma})$ where 
\[
H:=\bigoplus_{p \in \PP} H_p, \ \ \pi:=\bigoplus_{p \in \PP} \pi_p, \ \text{ and } \ \ D_{\sigma}:= \bigoplus_{p \in \PP} \sigma(p) D_p.
\] \label{spec_trip_P}
\end{enumerate}
\end{thm}

The remainder of the paper will be dedicated to proving Theorem \ref{nposspectrip}. To this end, we fix a set of prototiles $ \PP $, a nonperiodic substitution rule $ \omega $ on $ \PP $ with scaling factor $ \lambda > 1 $, and a fractal tree $ \mathfrak{F}_p $ in $ T_p $ for each $ p \in \PP $, with $ \kappa > 1 $ corresponding to the Perron-Frobenius eigenvalue of the edge substitution matrix. We note that the spectral triples constructed in \eqref{spec_trip_p} of Theorem \ref{nposspectrip}, are bounded below, and therefore, $D_p$ may be taken to be a positive operator (by suitably scaling the Perron-Frobenius eigenvector so that its minimal entry is greater than $0.5$). This is the primary reason for the more sophisticated spectral triple appearing in \eqref{spec_trip_P} of Theorem \ref{nposspectrip}.

We first concentrate on part \eqref{spec_trip_p} of Theorem \ref{nposspectrip}. Fix $p\in\PP$, and consider the $C^*$-algebra $A_{punc}$, the Hilbert space $ H_p = \ell^2 (T_p \setminus \{p\}) $, and the faithful representation $\pi_p $ of $ A_{punc} $ on $ H_p $ as defined in equation \eqref{eq:repn}, which we saw in the previous section. For $t,t' \in T_p$, recall the fractal distance $d _{\mathfrak{F}_p}(t,t')$ defined in Definition \ref{defn:fractal_distance}. We define $D_p$ on the canonical basis $ \{\delta_t : t \in T_p \setminus \{p\}\} $ of $H_p$ by
\[
D_p \delta_t := \ln(d_{\mathfrak{F}_p}(t,p)) \delta_t,
\]
where $t \in T_p \setminus \{p\}$, and we extend $ D_p $ by linearity to $\newspan\{\delta_t : t \in T_p \setminus \{p\}\}$. Since $p \not \in T_p \setminus \{p\}$, we have $d_{\mathfrak{F}_p}(t,p) \neq 0$ for all $t \in T_p \setminus \{p\}$, and hence, $\ln(d_{\mathfrak{F}_p}(t,p))$ is well-defined. Since $D_p \delta_t \in H_p$ for each $t \in T_p \setminus \{p\}$, $D_p$ is a densely defined operator on $H_p$. Furthermore, $D_p$ is an unbounded operator. To see this, fix a tile $t \in T_p \setminus \{p\}$, and $n \in \N$. Choose $N \in \N$ such that $d_{\mathfrak{F}_p}(t,p) \kappa^N > e^n$, and let $t'$ be the root of $\omega^N(t)$. Since $p$ is the root of $\omega^N(p)$, Lemma \ref{fscale} implies
\[
\ln(d_{\mathfrak{F}_p}(t',p)) = \ln(\kappa^N d_{\mathfrak{F}_p}(t,p)) > \ln(e^n) = n.
\]
Then, considering the operator norm of $ D_p $, we obtain
\begin{align*}
\|D_p\| &= \sup\{\|D_p \xi\| : \xi \in \Dom(D_p), \|\xi\| \leq 1 \} \geq \|D_p \delta_{t'}\| = \|\ln(d_{\mathfrak{F}_p}(t',p)) \delta_{t'}\| = \ln(d_{\mathfrak{F}_p}(t',p)) > n,
\end{align*}
so that $D_p$ is unbounded. Finally, we show that $D_p$ is self-adjoint. That $D_p$ is symmetric follows immediately from the definition of $D_p$ on $H_p$. To see that $D_p$ is self-adjoint, consider
\[
\Dom(D_p^*) = \{\eta \in H_p : \xi \mapsto \langle D_p \xi,\eta \rangle\ \text{from}\ \Dom(D_p) \to \C\ \text{is bounded}\}.
\]
Since $D_p$ is symmetric, we have $\Dom(D_p) \subset \Dom(D_p^*)$. We show that $\Dom(D_p^*) \subset \Dom(D_p)$. Let $\eta \in \Dom(D_p^*)$. Then the map $\xi \mapsto \langle D_p \xi,\eta \rangle$ from $\Dom(D_p)$ to $\C$ is bounded. By the Riesz Representation Theorem, there exists $\varphi \in H_p$ such that $\langle D_p \xi,\eta \rangle = \langle \xi,\varphi \rangle$ for all $\xi \in \Dom(D_p)$. For each $ t \in T_p \setminus \{p\} $, we have $ \delta_t \in \Dom (D_p) $, and hence,
\[
\langle D_p \delta_t,\eta \rangle = \langle \ln(d_{\mathfrak{F}_p}(t,p)) \delta_t,\eta \rangle = \ln(d_{\mathfrak{F}_p}(t,p)) \sum_{t' \in T_p \setminus \{p\}} \overline{\delta_t(t')} \eta(t') = \ln(d_{\mathfrak{F}_p}(t,p)) \eta(t).
\]
Furthermore, $ \langle \delta_t ,\varphi \rangle = \varphi (t) $, so that $\varphi(t) = \ln(d_{\mathfrak{F}_p}(t,p)) \eta(t)$. Thus, we obtain,
\[
\sum_{t \in T_p \setminus \{p\}} \ln(d_{\mathfrak{F}_p}(t,p))^2 |\eta(t)|^2 = \sum_{t \in T_p \setminus \{p\}} |\ln(d_{\mathfrak{F}_p}(t,p)) \eta(t)|^2 = \sum_{t \in T_p \setminus \{p\}} |\varphi(t)|^2 < \infty,
\]
so that $\eta \in \Dom(D_p)$. Thus, $\Dom(D_p^*) \subset \Dom(D_p)$, and hence, $D_p$ is self-adjoint. Given our unbounded and self-adjoint operator $D_p$ on $H_p$, we show that $(A_{punc},H_p,D_p)$ is a $ \theta $-summable spectral triple. To this end, we show that conditions \eqref{spectral triple 1} and \eqref{spectral triple 2} in Definition \ref{spectral_triple} are satisfied.

\begin{prop} \label{condition_3a}
Let $P$ be a patch in a tiling $T \in \Omega _{punc}$. Then, for all $t,t' \in P$,
\[
[D_p,\pi_p(e(P,t,t'))] \in \BB(H_p).
\]
\end{prop}

Proposition \ref{condition_3a} implies condition \eqref{spectral triple 1} in Definition \ref{spectral_triple}, since the linear span of functions $e(P,t,t')$ is dense in $A_{punc}$. We require a sequence of lemmas before arriving at the proof of Proposition \ref{condition_3a}. Given a patch $P$ in a tiling $T \in \Omega _{punc}$, we partition $T_p$ in such a way that we confine where $P$ can appear in $T_p$. The following lemma will help us to do this, and the corollary which follows describes what sort of control we have gained. First, we need the following definition.

\begin{definition}
For each $p \in \PP$ and $n \in \N$, the \emph{$n$th coronal radius of $p$} is the number:
\[
\corad_n(p) := \inf\{|x-y| : x \in \supp(\omega^n(p))\ \text{and}\ y \in \partial \supp(\omega^{n+1}(p)) \}.
\]
\end{definition}

\begin{remark}
	We note that for each $p \in \PP$, and $n \in \N$, we have $\corad_n(p) = \lambda^n \corad_0(p)$.
\end{remark}

\begin{lemma} \label{ballinsupertile}
Fix $p \in \PP$, and $r > 0$. There exists $N \in \N$ sufficiently large such that for all $x \in \R^2$, either $B(x,r) \subset\supp(\omega^{N+1}(p))$, or $B(x,r) \subset \supp(\omega^{k+2}(p)) \setminus \supp(\omega^k(p))$ for some $k \geq N$.
\end{lemma}

\begin{proof}
Let $x \in \R^2$. We may choose $N \in \N$ sufficiently large such that $\lambda^N \corad_0(p) > 2r$, since $\lambda > 1$. If $B(x,r) \subset \supp(\omega^{N+1}(p))$, then we are done. Otherwise, we either have,
\begin{enumerate}[(1)]
\item $B(x,r) \cap \supp(\omega^{N+1}(p)) \neq \varnothing$; or
\item $B(x,r) \cap \supp(\omega^{N+1}(p)) = \varnothing$.
\end{enumerate}

If situation (1) occurs, $B(x,r) \cap \partial \supp(\omega^{N+1}(p)) \neq \varnothing$. Since $\corad_N(p) > 2r$, we have $ B(x,r) \cap \supp(\omega^N(p)) = \varnothing $. Since $\corad_{N+1}(p) = \lambda \corad_N(p) > 2r$, we have $ B(x,r) \subset \supp(\omega^{N+2}(p)) $.

If situation (2) occurs, choose $k > N$ such that $B(x,r) \cap \supp(\omega^{k+1}(p)) \neq \varnothing$, and $B(x,r) \cap \supp(\omega^k(p)) = \varnothing$. Then, $\corad_{k+1}(p) = \lambda^{k-N} \corad_{N+1}(p) > 2r$ implies $B(x,r) \subset \supp(\omega^{k+2}(p))$.

In either case, there exists $k \geq N$ such that $B(x,r) \subset \supp(\omega^{k+2}(p)) \setminus \supp(\omega^k(p))$.
\end{proof}

\begin{cor} \label{patchinsupertile}
Let $P$ be a patch in $T_p$ for some $p \in \PP$. There exists $N \in \N$ sufficiently large such that either $P \subset \omega^{N+1}(p)$, or $P \subset \omega^{k+2}(p) \setminus \omega^k(p)$ for some $k \geq N$.
\end{cor}

For the following lemma, we recall the notation used Definition \ref{radii}.

\begin{lemma} \label{allofwork}
Let $P \subset T_p \setminus \{p\}$ be a patch in $T_p$, and $t,t' \in P$. There exists $N \in \N$ such that
\begin{equation}\label{inequality frac dist}
0 < \frac{d_{\mathfrak{F}_p}(t,p)}{d_{\mathfrak{F}_p}(t',p)} \leq \max \left\{ \frac{\LL_{N+1}(p)}{\SS},\frac{\kappa^3 \LL}{\SS(\kappa-1)} \right\}.
\end{equation}
\end{lemma}

\begin{proof}
By Corollary \ref{patchinsupertile}, there exists $N \in \N$ sufficiently large such that $P \subset \omega^{N+1}(p)$, or $P \subset \omega^{k+2}(p) \setminus \omega^k(p)$ for some $k \geq N$. Suppose that $P \subset \omega^{N+1}(p)$. By Definition \ref{radii}, we have
\begin{equation} \label{allofwork1}
0 < \frac{d_{\mathfrak{F}_p}(t,p)}{d_{\mathfrak{F}_p}(t',p)} \leq \frac{\LL_{N+1}(p)}{\SS_{N+1}(p)} \leq \frac{\LL_{N+1}(p)}{\SS}.
\end{equation}
Otherwise, $P \subset \omega^{k+2}(p) \setminus \omega^k(p)$ for some $k \geq N$. By definition of the $(k+2)$th minor branch radius of $p$ over $k$ in Definition \ref{radii}, the fact that $t \in \omega^{k+2}(p)$, and Lemmas \ref{longest} and \ref{shortcorona}, we have
\begin{equation} \label{allofwork2}
0 < \frac{d_{\mathfrak{F}_p}(t,p)}{d_{\mathfrak{F}_p}(t',p)} \leq \frac{\LL_{k+2}(p)}{\SS_{k+2,k}(p)} \leq \frac{\sum_{i=0}^{k+1} \LL \kappa^i}{\SS \kappa^{k-1}} = \frac{\LL}{\SS} \left (\kappa + \kappa^2 + \sum_{i=0}^{k-1} \kappa^{-i} \right ) \leq \frac{\kappa^3 \LL}{\SS (\kappa-1)}.
\end{equation}
Thus, combining equations \eqref{allofwork1} and \eqref{allofwork2}, we obtain
\[
0 < \frac{d_{\mathfrak{F}_p}(t,p)}{d_{\mathfrak{F}_p}(t',p)} \leq \max \left \{\frac{\LL_{N+1}(p)}{\SS},\frac{\kappa^3 \LL}{\SS (\kappa-1)} \right \}. \qedhere
\]
\end{proof}

Let $P$ be a patch in a tiling $T \in \Omega _{punc}$, $t,t' \in P$, and $t'' \in T_p \setminus \{p\}$. Making the necessary modifications to the equation immediately preceding \cite[Proposition 3.3]{Whi_rot}, we have the following formula for the representation of $e(P,t,t')$ in $\BB(H_p)$:
\begin{equation} \label{representation_delta}
\pi_p(e(P,t,t')) \delta_{t''} = \begin{cases}
\delta_{t''-(x(t')-x(t))} & \mbox{if}\ T_p-x(t'') \in U(P,t') \\
0 & \mbox{otherwise}
\end{cases}.
\end{equation}
It is important to note that if $t''-(x(t')-x(t)) = p$, then $\pi_p(e(P,t,t'))\delta_{t''} \equiv 0$. Given equation \eqref{representation_delta}, we are now in a position to prove Proposition \ref{condition_3a}.

\begin{proof}[Proof of Proposition \ref{condition_3a}]
Let $P$ be a patch in a tiling $T \in \Omega _{punc}$, $t,t' \in P$ and $t'' \in T_p \setminus \{p\}$. If $T_p - x(t'') \in U(P,t)$, let $s$ denote the tile $t''-(x(t')-x(t)) \in T_p$. We begin by calculating
\begin{align*}
[D_p,\pi_p(e(P,t,t'))] \delta_{t''} & = D_p \pi_p(e(P,t,t')) \delta_{t''}-\pi_p(e(P,t,t')) D_p \delta_{t''} \\
& = D_p \delta_s-\ln(d_{\mathfrak{F}_p}(t'',p)) \pi_p(e(P,t,t')) \delta_{t''} \\
& = \ln(d_{\mathfrak{F}_p}(s,p)) \delta_s-\ln(d_{\mathfrak{F}_p}(t'',p)) \delta_s \\
& = \ln \left[ \frac{d_{\mathfrak{F}_p}(s,p)}{d_{\mathfrak{F}_p}(t'',p)} \right] \delta_s.
\end{align*}
By equation \eqref{representation_delta}, if $T_p-x(t'') \not \in U(P,t')$ then $[D_p,\pi_p(e(P,t,t'))] \delta_{t''} \equiv 0$. Suppose then, that $T_p-x(t'') \in U(P,t')$. Thus, to prove that $[D_p,\pi_p(e(P,t,t'))] \in \BB(H_p)$, we show that
\[
\left| \ln \left[ \frac{d_{\mathfrak{F}_p}(s,p)}{d_{\mathfrak{F}_p}(t'',p)} \right] \right| \leq M,
\]
for some fixed $M \geq 0$. Using Lemma \ref{allofwork}, fix $N \in \N$ such that equation \eqref{inequality frac dist} holds. We claim
\[
M := \ln \left[ \max \left\{ \frac{\LL_{N+1}(p)}{\SS},\frac{\kappa^3 \LL}{\SS (\lambda-1)} \right\} \right],
\]
does the job. Suppose that $d_{\mathfrak{F}_p}(s,p) \geq d_{\mathfrak{F}_p}(t'',p)$. Then, by Lemma \ref{allofwork}, we have
\begin{equation} \label{eq:bound_M_1}
1 \leq \frac{d_{\mathfrak{F}_p}(s,p)}{d_{\mathfrak{F}_p}(t'',p)} \leq \max \left\{ \frac{\LL_{N+1}(p)}{\SS},\frac{\kappa^3 \LL}{\SS (\lambda-1)} \right\} \implies 0 \leq \ln \left[ \frac{d_{\mathfrak{F}_p}(s,p)}{d_{\mathfrak{F}_p}(t'',p)} \right] \leq M.
\end{equation}
Otherwise, $d_{\mathfrak{F}_p}(s,p) \leq d_{\mathfrak{F}_p}(t'',p)$ and again by Lemma \ref{allofwork}, we have
\begin{equation} \label{eq:bound_M_2}
1 \leq \frac{d_{\mathfrak{F}_p}(t'',p)}{d_{\mathfrak{F}_p}(s,p)} \leq \max \left\{ \frac{\LL_{N+1}(p)}{\SS},\frac{\lambda^3 \LL}{\SS (\lambda-1)} \right\} \implies 0 \leq \ln \left[ \frac{d_{\mathfrak{F}_p}(t'',p)}{d_{\mathfrak{F}_p}(s,p)} \right] \leq M.
\end{equation}
Thus, from equations \eqref{eq:bound_M_1} and \eqref{eq:bound_M_2}, we have
\[
\left| \ln \left[ \frac{d_{\mathfrak{F}_p}(t'',p)}{d_{\mathfrak{F}_p}(s,p)} \right] \right| = \left| \ln \left[ \frac{d_{\mathfrak{F}_p}(s,p)}{d_{\mathfrak{F}_p}(t'',p)} \right] \right| \implies 0 \leq \left| \ln \left[ \frac{d_{\mathfrak{F}_p}(s,p)}{d_{\mathfrak{F}_p}(t'',p)} \right] \right| \leq M,
\]
as claimed. Thus, for any patch $ P $ in any tiling $ T \in \Omega_{punc} $, and $ t,t' \in P $, we have
\[
[D_p,\pi_p(e(P,t,t'))] \in \BB(H_p). \qedhere
\]
\end{proof}

By Proposition \ref{condition_3a}, we now know that $(A_{punc},H_p,D_p)$ satisfies condition \eqref{spectral triple 1} of Definition \ref{spectral_triple}. For condition \eqref{spectral triple 2} of Definition \ref{spectral_triple}, we need to show that the operator $(1 + D_p^2) \inverse \in \KK (H_p)$. We show that $\exp(-\alpha D_p^2)$ is trace class for all $\alpha > 0$, and invoke the following result found in \cite[Section 7]{Con1}.

\begin{lemma}[{\cite[Section 7]{Con1}}] \label{thetatocompact}
Suppose $D : \Dom(D) \to H$ is an unbounded and self-adjoint operator on a separable Hilbert space $H$. If $\exp(-\alpha D^2)$ is trace class for all $\alpha > 0$, then $(1+D^2) \inverse \in \KK(H)$.
\end{lemma}

If we are able to show that $ \exp (-\alpha D_p^2) $ is trace class for all $ \alpha > 0 $, then Lemma \ref{thetatocompact} will imply that $ (A_{punc},H_p,D_p) $ is a $ \theta $-summable spectral triple, and hence, we will be able to prove part \eqref{spec_trip_p} of Theorem \ref{nposspectrip}.
Before proceeding, we need the following technical lemma.

\begin{lemma} \label{leftfield}
For each $n \in \N$, there exists a constant $c > 0$ such that $x^{\ln x} \geq c x^n$ for all $x > 0$.
\end{lemma}

\begin{proof}
Since $x>0$, the property $x^{\ln x} \geq c x^n$ is equivalent to $(\ln x)^2 -n\ln x \geq \ln c$. Now, $(\ln x)^2 -n\ln x$ has a lower bound of $-n^2/4$ at $x=e^{n/2}$. Thus, $c=e^{-n^2/4}$ is the desired constant.
\end{proof}

In the following proposition, we recall the notation given introduced in Notation \ref{notation:substitution} for the number of tiles comprising the patch $ \omega^n (p) $ for each $ n \in \N $ and $ p \in \PP $.

\begin{prop} \label{thetasum}
The operator $\exp(-\alpha D_p^2)$ is trace class for all $\alpha > 0$.
\end{prop}

\begin{proof}
Fix $\alpha > 0$. For any $t \in T_p \setminus \{p\}$, we have
\[
\exp(-\alpha D_p^2) \delta_t = \exp(-\alpha (\ln(d_{\mathfrak{F}_p}(t,p)))^2) \delta_t = d_{\mathfrak{F}_p}(t,p)^{-\alpha \ln(d_{\mathfrak{F}_p}(t,p))} \delta_t.
\]
From this, we calculate
\[
\Tr(\exp(-\alpha D_p^2)) = \sum_{t \in T_p \setminus \{p\}} \langle \exp(-\alpha D_p^2) \delta_t,\delta_t \rangle
 = \sum_{t \in T_p \setminus \{p\}} (d_{\mathfrak{F}_p}(t,p)^{\ln(d_{\mathfrak{F}_p}(t,p))})^{-\alpha}.
\]
Since $\kappa > 1$ and $\alpha > 0$, we have $\kappa^{\alpha} > 1$. Fix $N \in \N$ such that $\NN_{\max} < \kappa^{N\alpha}$. By Lemma \ref{leftfield}, there exists $c > 0$ such that $x^{\ln x} \geq c x^N$ for all $x > 0$. Since $d_{\mathfrak{F}_p}(t,p) > 0$ for all $t \in T_p \setminus \{p\}$, we have
\begin{equation} \label{summationchange}
\Tr(\exp(-\alpha D_p^2)) \leq c^{-\alpha} \sum_{t \in T_p \setminus \{p\}} d_{\mathfrak{F}_p}(t,p)^{-N\alpha}.
\end{equation}
From here, let us restrict our attention to the sum in equation \eqref{summationchange}, which we rewrite as follows. For each $t \in T_p \setminus \{p\}$, we have $t \in \omega(p) \setminus \{p\}$, or $t \in \omega^{2m+1}(p) \setminus \omega^{2m-1}(p)$ for some $m \in \N$. To simplify notation, define $\omega^{2m\pm 1}(p) := \omega^{2m+1}(p) \setminus \omega^{2m-1}(p)$ for each $m \in \N$. Thus,
\begin{equation} \label{tilesummation}
\sum_{t \in T_p \setminus \{p\}} d_{\mathfrak{F}_p}(t,p)^{-N\alpha} = \sum_{t \in \omega(p) \setminus \{p\}} d_{\mathfrak{F}_p}(t,p)^{-N\alpha} + \sum_{m=1}^{\infty} \sum_{t \in \omega^{2m\pm 1}(p)} d_{\mathfrak{F}_p}(t,p)^{-N\alpha}.
\end{equation}
The first sum on the right hand side of equation \eqref{tilesummation} is finite since $\omega(p) \setminus \{p\}$ is a patch. Thus, we only concern ourselves with the second sum on the right hand side of equation \eqref{tilesummation}. For each $m \in \N$, Lemma \ref{shortcorona}, with $n = 2m-1$, implies that $d_{\mathfrak{F}_p}(t,p) \geq \kappa^{2m-2} \SS$ for any $t \in \omega^{2m\pm 1}(p)$. Thus,
\[
\sum_{m=1}^{\infty} \sum_{t \in \omega^{2m\pm 1}(p)} d_{\mathfrak{F}_p}(t,p)^{-N\alpha} \leq \sum_{m=1}^{\infty}\sum_{t \in \omega^{2m \pm 1}(p)} (\kappa^{2m-2} \SS)^{-N\alpha} = (\kappa^{-2} \SS)^{-N\alpha} \sum_{m=1}^{\infty} \sum_{t \in \omega^{2m\pm 1}(p)} \kappa^{-2mN\alpha}.
\]
For each $m \in \N$, the sum over $t \in \omega^{2m\pm 1}(p)$ no longer depends on $t$. Thus, we multiply $\kappa^{-2mN\alpha}$ by the number of tiles in $\omega^{2m\pm 1}(p)$, which is certainly smaller than $ \NN_{\max}^{2m + 1} $. Thus, we have
\[
\sum_{m=1}^{\infty} \sum_{t \in \omega^{2m\pm 1}(p)} d_{\mathfrak{F}_p}(t,p)^{-N\alpha} \leq (\kappa^{-2} \SS)^{-N\alpha} \sum_{m=1}^{\infty} \NN_{\max}^{2m+1} \kappa^{-2mN\alpha} = \NN_{\max}(\kappa^{-2} \SS)^{-N\alpha} \sum_{m=1}^{\infty} \left (\frac{\NN_{\max}^2}{\kappa^{2N\alpha}}\right )^m.
\]
By our choice of $N$, we have $\NN_{\max}^2 < \kappa^{2N\alpha}$, so we have a convergent geometric series. Tracing back to equation \eqref{summationchange}, we have $\Tr(\exp(-\alpha D_p^2)) < \infty$. Thus, $\exp(-\alpha D_p^2)$ is trace class, as required.
\end{proof}

We are now in a position to prove part \eqref{spec_trip_p} of Theorem \ref{nposspectrip}.

\begin{proof}[Proof of Theorem \ref{nposspectrip} \eqref{spec_trip_p}]
The operator $D_p$ is an unbounded and self-adjoint operator on $H_p$, which satisfies condition \eqref{spectral triple 1} of Definition \ref{spectral_triple}, by Proposition \ref{condition_3a}. Moreover, since the operator $\exp(-\alpha D_p^2)$ is trace class for all $\alpha > 0$, by Proposition \ref{thetasum}, Lemma \ref{thetatocompact} implies that $(1+D_p^2)^{-1} \in \KK (H_p)$, which is condition \eqref{spectral triple 2} of Definition \ref{spectral_triple}. Thus, $(A _{punc},H_p,D_p)$ is a $\theta$-summable spectral triple.
\end{proof}

We now turn our attention to part \eqref{spec_trip_P} of Theorem \ref{nposspectrip}. By part \eqref{spec_trip_p} of Theorem \ref{nposspectrip}, we have a collection of spectral triples $(A _{punc},H_p,D_p)$ for each $p \in \PP$. Given $ \sigma : \PP \to \{-1,1\} $, define
\[
H:=\bigoplus_{p \in \PP} H_p, \qquad \pi:=\bigoplus_{p \in \PP} \pi_p,\,\, \text{ and } \qquad D_{\sigma}:= \bigoplus_{p \in \PP} \sigma(p) D_p.
\]
The map $\pi : A_{punc} \to \BB (H)$ is a non-degenerate faithful representation since $\pi_p$ is a non-degenerate faithful representation for each $p \in \PP$. Thus, to show that $(A_{punc},H,D _{\sigma})$ is a spectral triple, we check $D _{\sigma}$ is an unbounded, self-adjoint operator satisfying conditions \eqref{spectral triple 1} and \eqref{spectral triple 2} of Definition \ref{spectral_triple}.

\begin{lemma}
The operator $D _{\sigma}$ is an unbounded, self-adjoint operator on $H$ with
\begin{equation} \label{domD}
\Dom(D _{\sigma}) = \bigoplus_{p \in \PP} \Dom(D_p).
\end{equation}
\end{lemma}

\begin{proof}
Since $\| D _{\sigma} \| = \max\{ \| D_p \| : p \in \PP \}$ and each $D_p$ is unbounded on $H_p$, $D _{\sigma}$ is an unbounded operator on $H$. Now, \cite[Lemma 5.3.7]{Pedersen.book} implies that $D _{\sigma}$ is self-adjoint, and that \eqref{domD} holds.
\end{proof}

\begin{lemma} \label{cond_3a_D}
Let $P$ be a patch in a tiling $T \in \phull$. Then, for all $t,t' \in P$,
\[
[D _{\sigma},\pi(e(P,t,t'))] \in \BB(H).
\]
\end{lemma}

\begin{proof}
For each $ p \in \PP $, we have $ [D_p,\pi_p (e (P,t,t'))] \in B (H_p) $. Thus,
\begin{align*}
[D _{\sigma},\pi(e(P,t,t'))] & = D _{\sigma} \pi(e(P,t,t')) - \pi(e(P,t,t')) D _{\sigma} \\
& = \bigoplus_{p \in \PP}\ \sigma(p) D_p \pi_p(e(P,t,t')) - \sigma(p) \pi_p(e(P,t,t')) D_p \\
& = \bigoplus_{p \in \PP}\ \sigma(p) [D_p,\pi_p(e(P,t,t'))] \in \BB(H). \qedhere
\end{align*}
\end{proof}

\begin{lemma} \label{trace_class}
The operator $\exp(-\alpha D _{\sigma}^2)$ is trace class for all $\alpha > 0$.
\end{lemma}

\begin{proof}
Proposition \ref{thetasum} implies that $\Tr(\exp(-\alpha D_p^2)) < \infty$ for each $\alpha > 0$ and $p \in \PP$. Fix $\alpha > 0$. Then, by the holomorphic functional calculus we obtain
\[
\Tr(\exp(-\alpha D _{\sigma}^2)) = \Tr(\oplus_{p \in \PP} \exp(-\alpha D_p^2)) = \sum_{p \in \PP} \Tr(\exp(-\alpha D_p^2)) < \infty.
\]
Since $\alpha > 0$ was arbitrary, the operator $\exp(-\alpha D _{\sigma}^2)$ is trace class for all $\alpha > 0$.
\end{proof}

\begin{proof}[Proof of Theorem \ref{nposspectrip} \eqref{spec_trip_P}]
The operator $D _{\sigma}$ is an unbounded and self-adjoint operator on $H$, which satisfies condition \eqref{spectral triple 1} of Definition \ref{spectral_triple}, by Lemma \ref{cond_3a_D}. Moreover, since the operator $\exp(-\alpha D _{\sigma}^2)$ is trace class for all $\alpha > 0$, by Lemma \ref{trace_class}, Lemma \ref{thetatocompact} implies that $(1+D _{\sigma}^2)^{-1} \in \KK (H)$, which is condition \eqref{spectral triple 2} of Definition \ref{spectral_triple}. Thus, $(A _{punc},H,D _{\sigma})$ is a $\theta$-summable spectral triple.
\end{proof}

\begin{example}
Recall the fractal trees constructed for the Penrose tiling in Example \ref{eg:fractaltree}. Theorem \ref{nposspectrip} implies that we have a spectral triple $(A_{punc},H,D_{\sigma})$ on $A_{punc}$ for each function $\sigma:\PP \to \{-1,1\}$.
\end{example}

We conclude the paper with two open questions:
\begin{enumerate}
\item Suppose $(A_{punc},H,D_{\sigma})$ is a spectral triple on $A_{punc}$ associated with a collection of fractal trees $\{\mathfrak{F}_p \mid p \in \PP\}$, as in Theorem \ref{nposspectrip} \eqref{spec_trip_P}. Under what conditions will $(A_{punc},H,D_{\sigma})$ be a \emph{quantum metric space} in the sense of Rieffel \cite{Rie}? If $(A_{punc},H,D_{\sigma})$ is a quantum metric space, then when does this property extend to the crossed product $A_{punc} \rtimes \Z$? We note that $A_{punc} \rtimes \Z$ is a Ruelle algebra, as defined by Putnam in \cite{Put1}. The paper of Bellissard, Marcolli, and Reihani \cite{BMR} is likely to factor into the resolution of this question.

\item Suppose $T$ is tiling satisfying the hypotheses of \cite[Theorem 6.6]{FWW} (see also Theorem \ref{thm6.6-FWW}). We conjecture that there exists a collection of fractal trees $\{\mathfrak{F}_p \mid p \in \PP\}$ such that the $K$-homology group $KK^1(A_{punc},\C)$ is generated by $\{[D_{\sigma}(1+D_{\sigma}^2)^{-1/2}] \mid \, \sigma:\PP \to \{-1,1\}\}$, where each such $\sigma$ defines a spectral triple $(A_{punc},H,D_{\sigma})$ as in Theorem \ref{nposspectrip} \eqref{spec_trip_P}.
\end{enumerate}

\end{document}